\documentclass[11pt,reqno]{amsproc}
\usepackage[margin=1in]{geometry}
\usepackage{amsmath, amsthm, amssymb}
\usepackage{hyperref}
\usepackage[abbrev,lite,nobysame]{amsrefs}
\usepackage{times}
\usepackage[usenames,dvipsnames]{color}
\usepackage{bbm}
\usepackage{bm}
\usepackage{subcaption}
\usepackage{mathtools}
\usepackage{pdfpages}
\usepackage{comment}

\newcommand{\R}{\mathbb{R}}

\newcommand{\N}{\mathbb{N}}

\newcommand{\C}{\mathbb{C}}

\newcommand{\A}{\mathcal{A}}
\newcommand{\B}{\mathcal{B}}
\newcommand{\cC}{\mathcal{C}}

\newcommand{\dd}{\mathrm{d}}

\newcommand{\de}{\partial}

\newcommand{\ue}{u^\varepsilon}

\newcommand{\wh}[1]{\widehat{#1}}
\newcommand{\wt}[1]{\widetilde{#1}}

\providecommand{\R}{\mathbb{R}}
\providecommand{\C}{\mathbb{C}}
\providecommand{\N}{\mathbb{N}}

\providecommand{\eps}{\varepsilon}

\renewcommand{\leq}{\leqslant}
\renewcommand{\geq}{\geqslant}

\renewcommand{\Im}{\mbox{Im}}

\DeclareMathOperator{\sgn}{sgn}

\newcommand{\Le}{\mathcal L^\eps}

\DeclareMathOperator{\sech}{sech}

  \newtheorem{thm}{Theorem}[section]
  
  \newtheorem{Lemma}[thm]{Lemma}
  
  \theoremstyle{definition}
    
  \theoremstyle{remark}
  \newtheorem{remark}[thm]{Remark}

\usepackage{etoolbox}
\patchcmd{\subsubsection}{\itshape}{\itshape\bfseries}{}{} 

\title[]{Cost of controllability of the Burgers’ equation linearized at a steady shock in the vanishing viscosity limit}

\author[V.\ Laheurte]{Vincent Laheurte}
\address{DMATH, Université du Luxembourg}
\email{vincent.laheurte@uni.lu}

\begin{document}
\maketitle
\begin{abstract}
    We consider the one-dimensional Burgers' equation linearized at a stationary shock, and investigate its null-controllability cost with a control at the left endpoint. We give an upper and a lower bound on the control time required for this cost to remain bounded in the vanishing viscosity limit, and construct an admissible control with an explicit limit behavior. We also provide an extension of the analysis to the case where the control acts on both endpoints. The proof relies on complex analysis and adapts methods previously used to tackle the same issue with a constant transport term. 
\end{abstract}

\setcounter{tocdepth}{3}

\let\oldtocsection=\tocsection

\let\oldtocsubsection=\tocsubsection

\let\oldtocsubsubsection=\tocsubsubsection

\tableofcontents

\section{Introduction}
\subsection{Shock profiles for Burgers' equation}
When considering weak solutions for the one-dimensional inviscid Burgers' equation \begin{equation}\left\{\begin{aligned}
    \de_tu(t,x)+u(t,x)\de_xu(t,x)&=0,\quad& t\in\R^+, x\in\R,\\
    u(0,x)&=u_0(x),\quad x\in\R,
    \end{aligned}\right.
\end{equation}
shocks may arise in finite time regardless of the initial datum's regularity. Indeed, by the method of characteristics, if $u_0$ is at least $C^1$ and there exists some $x\in\R$ such that $u_0'(x)<0$, two distinct characteristics will meet in time \[T=-\frac1{\inf_{x\in\R}u_0'(x)},\]which causes a jump discontinuity to appear in the solution $u$ for times greater than $T$.\\
Such shocks are jumps from $u^-$ to $u^+$, with $u^->u^+$, and propagate, by the Rankine-Hugoniot condition, at a speed \[s=\frac{u^-+u^+}{2}.\]
Stationary shocks therefore exist and are jump discontinuities from a positive value to its opposite.\\
To study the shocks, we may restrain the equation to an interval $[-L,L]$, $L>0$ and impose suitable Dirichlet conditions on the endpoints: \begin{equation}\left\{\begin{aligned}\label{shocks}
    \de_tu+u\de_xu&=0,\\
    u(t,-L)&=1,\\
    u(t,L)&=-1,\\
    u(0,x)&=u_0(x).
    \end{aligned}\right.
\end{equation}
The meaning of this system as well as the treatment of the boundary values are delicate, and we will consider \textit{weak entropy solutions}, for which we refer to \cite{bardos1979first}. In this setting, this system admits infinitely many stationary solutions which, from the Rankine-Hugoniot condition, are all of the form (see \cites{folino2017metastability,mascia2013metastability}) \[U_\sigma(x)=\begin{cases}
    1&\text{if }x<\sigma,\\
    -1&\text{if }x>\sigma.
\end{cases}\]
Moreover, if the initial datum $u_0$ has bounded variations, then the corresponding solution converges in finite time to one of these stationary solutions, see \cite{mascia2013metastability}.\\
If the initial datum has low regularity, the system becomes ill-posed, as initial conditions with a jump discontinuity from a smaller to a greater value lead to infinitely many weak solutions. To avoid this issue and define the "right" unique solution, the approach first used by Hopf \cite{hopf} is to add a small viscosity term $\eps\de_x^2$, whose regularizing effect makes the system well-posed, and let $\eps$ tend to 0. We are therefore interested in weak solutions of the viscous system

\begin{equation}\left\{\begin{aligned}\label{13}
    \de_t \ue+\ue\de_x\ue&=\eps\de_x^2\ue,\\
    \ue(t,-L)&=1,\\
    \ue(t,L)&=-1.
    \end{aligned}\right.
\end{equation}
Contrarily to the inviscid system \eqref{shocks}, by solving a second order ODE, this one admits a unique stationary solution. This solution has the form of a fast transition layer from $1$ to $-1$, resembling the inviscid shock $U_0$. This stationary solution is therefore referred to as a stationary viscous shock, or stationary shock by an abuse of language. This viscous shock profile is given by \begin{equation}\label{Ue}U^\eps(x)=-\kappa^\eps\tanh\left(\frac{\kappa^\eps x}{2\eps}\right),\end{equation} where $\kappa^\eps>0$ is such that $U^\eps(\pm L)=\mp1$.\\ This constant exists and is unique, as the function $\kappa\mapsto \kappa\tanh\left(\frac{\kappa L}{\eps}\right)$ is increasing on $\R^+$ and goes to $+\infty$ as $\kappa\to+\infty$. By the intermediate value theorem, this unique solution $\kappa^\eps$ verifies \[1<\kappa^\eps<\frac{1}{\tanh\left(\frac L\eps\right)},\]and is therefore exponentially close to $1$ as $\eps$ goes to $0$.\newline\newline
The $L^2$-stability of such viscous shocks in the context of Burgers equation was first studied by Il’in and Oleinik \cite{il1960asymptotic} by a maximum principle, and later by Sattinger \cite{sattinger1976stability} via spectral analysis tools. We refer to \cites{kawashima1994stability, matsumura1994asymptotic, matsumura1997nonlinear} for $L^2$-stability results with other flux functions $f(u)$, convex or non-convex.\\
 We wish to study some controllability aspects of the viscous Burgers' equation linearized at the shock profile \eqref{Ue}.

\subsection{Vanishing viscosity controllability problem}

We study the left-side null-controllability of the viscous Burgers' equation linearized around a stationary shock. Namely, we set an arbitrary control time $T>0$, and consider the control system
\begin{equation}\label{cont}\left\{\begin{aligned}
    \de_t \ue+\de_x(U^\eps \ue)&=\eps\de_x^2\ue,\quad &x\in (-L,L), t\in(0,T),\\
    \ue(t,-L)&=h^\eps(t),\quad &t\in(0,T),\\
    \ue(t,L)&=0,\quad &t\in (0,T),\\
    \ue(0,x)&=u_0,\quad &x\in(-L,L),
    \end{aligned}\right.
\end{equation}
where the initial datum $u_0$ lies in $L^2(-L,L)$, and the shock profile $U^\eps$ is given by \[U^\eps(x)=-\tanh\left(\frac{x}{2\eps}\right),\] where we ignore the $\kappa^\eps$ from \eqref{Ue} as it does not affect the calculations in a meaningful way.\\

For $\sigma\in (-L,L)$, we will more generally study the uncentered control problem\begin{equation}\label{cont asym}\left\{\begin{aligned}
    \de_t \ue+\de_x(U_\sigma^\eps \ue)&=\eps\de_x^2\ue,\quad &x\in (-L,L), t\in(0,T),\\
    \ue(t,-L)&=h^\eps(t),\quad &t\in(0,T),\\
    \ue(t,L)&=0,\quad &t\in (0,T),\\
    \ue(0,x)&=u_0,\quad &x\in(-L,L),
    \end{aligned}\right.
\end{equation}
where the shock profile is \[U^\eps_\sigma(x):=U^\eps(x-\sigma).\]
Although the viscous shock profiles $U^\eps_\sigma$ aren't solutions of the nonlinear problem \eqref{13}, for $\sigma\ne0,$ due to the boundary conditions not being met, we study them as they provide a good viscous counterpart to the stationary shocks $U_\sigma$.\\

Given $\eps>0, T>0, \sigma\in(-L,L)$, the issue of null-controllability is, for any initial datum, to find a control $h^\eps\in L^2(0,T)$ such that the solution of \eqref{cont asym} satisfies $\ue(T)=0$. If this is possible, we may define the null-controllability cost: \[\cC(T,L,\sigma,\eps):=\sup_{u_0\in L^2, \|u_0\|_{L^2}=1} \inf \left\{\|h^\eps\|_{L^2(0,T)}: \ue(T)=0\right\}.\]

This type of problems has been deeply studied \cite{fattorini1974uniform} and the system \eqref{cont} can be proved to be null-controllable for any $\eps>0$ and $T>0$. In this paper, we will tackle the issue of uniform null-controllability in the vanishing viscosity limit. Namely, we want to find the minimal time $T_{\rm{unif}}$ such that, for any $T>T_{\rm{unif}}$, the controllability cost $\cC(T,L,\sigma,\eps)$ remains bounded as $\eps\to0$.

\subsection{Control of the limit system}\label{sec lim sys}

Formally taking $\eps=0$ in the system \eqref{cont asym}, the limit control problem we obtain can be written as \begin{equation}\label{lim}
    \left\{\begin{aligned}
    \de_t u-\de_x(\sgn(x-\sigma)u)&=0,\\
    u(t,-L)&=h(t),\\
    u(t,L)&=0,\\
    u(0,x)&=u_0(x).
    \end{aligned}\right.
\end{equation}
We are allowed to impose boundary conditions on both endpoints as the transport $-\sgn(x-\sigma)$ is entering on the boundary, namely $-\sgn(-L-\sigma)>0, -\sgn(L-\sigma)<0.$ Solutions are to be interpreted in the sense of transposition, that is for any $t\in [0,T], \phi\in C^1([0,t]\times[-L,L]):$ \[\int_{-L}^L u(t,x)\phi(t,x)\,\dd x-\int_{-L}^L u_0(x)\phi(0,x)=\int_0^t u(\de_t \phi-\sgn(x-\sigma)\de_x\phi) + \int_0^th(\tau)\phi(\tau,-L)\,\dd \tau.\]
By the method of characteristics, for any time $t>0$, the solution $u(t)$ of this system is given by
\begin{align}
    \forall x \in(-L,\sigma),&\quad u(t,x)=\begin{cases}u_0(x-t)& t<x+L\\ h(t-x-L)&t\ge x+L \end{cases},\\
    \forall x\in(\sigma,L),&\quad u(t,x)=\begin{cases}u_0(x+t)& t<L-x\\ 0&t\ge L-x \end{cases},
\end{align}
and the solution has a Dirac in $x=\sigma$ of mass \begin{equation}
    m(t)=\begin{cases}\int_{\sigma-t}^{\sigma+t} u_0(x)\,\dd s,& t<L-\sigma\\ \int_{\sigma-t}^L u_0(x)\dd x, &L-\sigma\le t\le L+\sigma\\\int_{-L}^L u_0(x)\,\dd x+\int_0^{t-L-\sigma} h(s)\,\dd s, &t\ge L+\sigma \end{cases},
\end{equation}
if $\sigma\ge0$, and 
\begin{equation}
    m(t)=\begin{cases}\int_{\sigma-t}^{\sigma+t} u_0(x)\,\dd s,& t<L+\sigma\\ \int_{-L}^{\sigma+t} u_0(x)\dd x+\int_0^{t-L-\sigma}h(s)\,\dd s, &L+\sigma\le t\le L-\sigma\\\int_{-L}^L u_0(x)\,\dd x+\int_0^{t-L-\sigma} h(s)\,\dd s, &t\ge L+\sigma \end{cases},
\end{equation}
if $\sigma\le0$.\\
From this expression, one easily obtains that the system \eqref{lim} is only null-controllable if $T>L+|\sigma|$, and $u(T)\equiv 0$ if and only if the two following constraints on $h$ are verified: \begin{align}
    h(t)=0,\quad t\ge T-L-\sigma,\\
    \int_0^{T-L-\sigma} h(t)\,\dd t+\int_{-L}^Lu_0(x)\,\dd x=0.
\end{align}
We then immediately conclude that the optimal control associated to an initial datum $u_0$ is given by \begin{equation}\label{contlim}h(t)=\begin{cases}
    -\frac1{T-L-\sigma}\int_{-L}^L u_0(x)\,\dd x,&\text{if } t<T-L-\sigma,\\
    0&\text{if }t\ge T-L.
\end{cases}\end{equation}

\begin{figure}[h]
\begin{tabular}{cc}
	\includegraphics[width=0.45\textwidth]{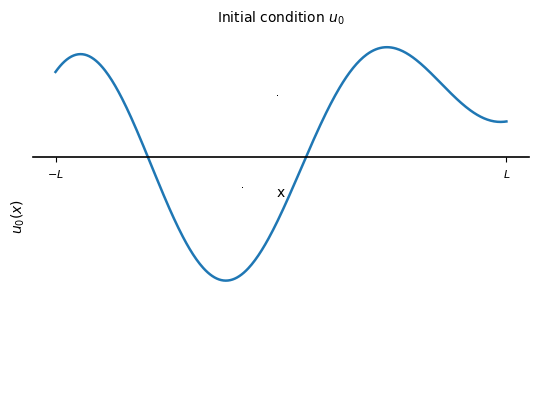} &
	\includegraphics[width=0.45\textwidth]{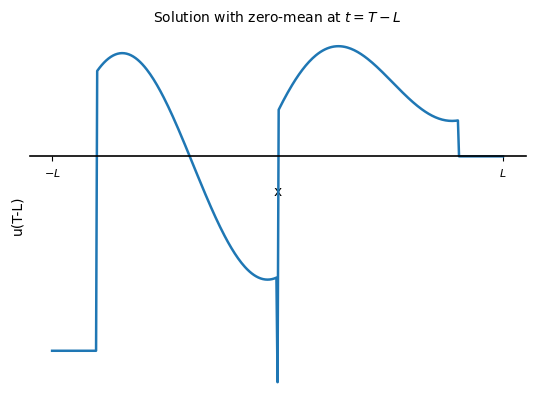} \\
\end{tabular}\caption{Evolution of the  solution with this choice of control, for $\sigma=0$.}
\end{figure}
Applying Cauchy-Schwarz' inequality, the null-controllability cost of the limit system is therefore given by \begin{equation}
    \cC(T,0)=\frac{\sqrt{2L}}{\sqrt{T-L-\sigma}}.
\end{equation}
It may then be expected that the viscous problem \eqref{cont} is uniformly null-controllable for times $T>L+|\sigma|$, and that its null-controllability cost converges to \(\lim_{\eps\to0}\cC(T,L,\sigma,\eps)=\cC(T,L,\sigma,0)=\frac{\sqrt{2L}}{\sqrt{T-L-\sigma}}\).

\subsection{Related problems}

This kind of problem was first addressed by Coron and Guerrero \cite{CG}, where they consider a transport-diffusion system on a finite interval $[0,L]$, with a constant transport term:
\begin{equation}\label{CG}
    \left\{\begin{aligned}
        \de_tu+M\de_x u&=\eps\de_x^2u,\\
        u(t,0)&=h(t),\\
        u(t,L)&=0.
    \end{aligned}\right.
\end{equation}
This system highlights different behaviors depending on the sign of the transport $M$. By using a dissipation argument and Carleman estimates, Coron and Guerrero \cite{CG} obtained a first upper-bound on uniform null-controllability time, which was later improved by Glass \cite{glass2010complex} through a method of moments, and by Lissy \cite{lissy2012link} with a reduction to the control of heat equation in short time. This result was later improved by Dardé-Ervedoza \cite{darde2019cost} with sharper estimates on the heat equation. In \cite{lissy2015explicit}, Lissy also provided a non-trivial lower bound on the uniform null-controllability time, with complex-analytic tools.\newline\newline
Similar results were then obtained for non-constant transport terms $M(t,x)$ in several space dimensions by Guerrero and Lebeau \cite{guerrero2007singular}, only proving the existence of a uniform time, without any estimate on it. Laurent and Léautaud \cites{laurent2016uniform, laurent2023uniform} later provided an upper bound on the uniform control time in similar contexts.\\
In the nonlinear setting, some partial results were given by Glass and Guerrero \cite{GG} for the Burgers' equation, and extended by Léautaud \cite{leautaudnonlin} to other conservation laws, by first bringing the solution to the neighborhood of a traveling wave.

\subsection{Main results}
 
We state below the main results, giving the existence and constraints on the control time required for the null-controllability cost to remain bounded in the vanishing viscosity limit.

\begin{thm}\label{thm1}
    For every $\sigma\in(-L,L),$, there exists a minimal time $T_{\rm{unif}} < +\infty$ such that the system \eqref{cont asym} is uniformly null-controllable for any time $T>T_{\rm{unif}}$. If $\sigma\ge 0$, this minimal time verifies \begin{equation}
        \label{tunif} T_{\rm{unif}}\in\left[(4\sqrt2-2), 4\sqrt3\right]L.
    \end{equation}
    Otherwise, if $\sigma<0$, it instead verifies \begin{equation}
        \label{tunif2} (4\sqrt2-2)L+8|\sigma| \le T_{\rm{unif}}\le 4(2|\sigma|+\sqrt{4\sigma^2+3L^2}).
    \end{equation}
\end{thm}

 \begin{remark}
     The upper bound on the uniform null-controllability time for $\sigma\le0$ matches the one obtained in \cite{lissy2012link} with $M=1$, as we work on an interval of length $2L$. The lower bound is however slightly worse than the one from \cite{lissy2015explicit} due to the behavior of the eigenfunctions of our operator at the left endpoint. We insist on the fact that, similarly to the Coron-Guerrero case with negative velocity, there is a time regime where the limit system is null-controllable, but the cost of the viscous system blows up as $\eps$ goes to $0$. 
 \end{remark}
 
We also provide a description of some admissible control function $h^\eps$ for times large enough.

\begin{thm}\label{thm2}
    Let \begin{equation}\label{def t*}T^*:=\begin{cases}
        4\sqrt3L,&\sigma\ge 0,\\
        4(2|\sigma|+\sqrt{4\sigma^2+3L^2}),& \sigma<0.
    \end{cases}\end{equation} Let $T>T^*$. Then, there exists $C>0$ such that, for any $u_0\in L^2(-L,L)$ and any $\eps>0$, there exists a control function $h_{u_0}^\eps\in L^2(0,T)$ such that the weak solution $\ue$ of \eqref{cont asym} satisfies $\ue(T)=0$ and \begin{equation} \label{major}
        \|h_{u_0}^\eps\|_{L^2(0,T)}\le \left(\frac{4\sqrt{L}}{\sqrt{T-T^*}}+Ce^{-\frac C\eps}\right)\|u_0\|_{L^2(-L,L)}.
    \end{equation}
    As $\eps$ goes to $0$, this control function converges in $L^2(0,T)$ towards the limit control \begin{equation}\label{7}
        h_{u_0}^0(t)=\begin{cases}
            -\frac2{T-T^*}\int_{-L}^Lu_0(x)\,\dd x,& t\le\frac{T-T^*}{2},\\
            0& t>\frac{T-T^*}2.
        \end{cases}
    \end{equation}
    More precisely, there exists a constant $c>0$ independent of the paramteters such that, for $\eps$ small enough, \begin{equation}\label{111}
        \|h_{u_0}^\eps-h_{u_0}^0\|_{L^2(0,T)}\le c\frac{\eps\|u_0\|_{L^2(-L,L)}}{\sqrt{T-T^*}}.
    \end{equation}
\end{thm}

The proof of 
$T_{\rm{unif}}\leq 4\sqrt3 L$ and of 
\eqref{major}
is given in 
Section 
 \ref{sec-upper}. It uses two intermediate results which are proved, respectively in Section 
\ref{sec-lem1} and in Section 
\ref{sec-lem2}. 
The proof of 
$T_{\rm{unif}}\geq (4\sqrt2-2)L$ is given in 
Section 
 \ref{sec-lissy}.
 These results rely on the spectral analysis of the operator at stake in the system \eqref{cont}  which is performed in Section 
\ref{sec-spec}.
The proof of Theorem \ref{thm2} is done within the proof of the upper bound, as we give a constructive proof of the control cost. More precisely, the estimate \eqref{major} is proven in Section \ref{48}, and the limit behavior \eqref{7} is given in \eqref{71} and \eqref{72}.\\
An extension of this result to the case with two boundary controls is given in Section \ref{sec ts}.

\begin{remark}
    The admissible control built in Theorem \ref{thm2} has the same structure as the optimal control of the limit system \eqref{contlim}. It morally cancels out the mean of the solution $\ue$ in an arbitrarily small time and then exploits the strong dissipation of the system.
\end{remark}

It is currently unclear whether these results may improve or prove further controllability results in the non-linear case, extending the work of \cite{GG}, or if similar results can be obtained for other conservation laws or more generally for systems with other singularities.

\section{Spectral analysis of the operator}
\label{sec-spec}
We are interested in the eigenvalues and eigenfunctions of the linearized Burgers operator \[\Le_\sigma u(x) :=\de_x\left(U_\sigma^\eps(x)u(x)\right) - \eps\de_x^2u(x),\]acting on $H_0^1\cap H^2(-L,L)$.\\
In the centered case $\sigma=0$, this operator, among other viscous conservation laws, has been studied in \cite{kreiss1986convergence}, with a deeper look at the metastability phenomenon in \cites{folino2017metastability, mascia2013metastability}. In these papers, the authors highlighted that the eigenvalues $(\lambda_{0,k}^\eps)_{k\ge0}$ of the operator $\Le_0$ are simple, real, positive, and are distributed as follows: \[\lambda_{0,0}^\eps=\mathcal O\left(\exp\left(-\frac1{\eps}\right)\right),\quad \lambda_{0,k}^\eps>\frac{1}{4\eps}, k\ge1.\]

The first notable difference with the operator of the Coron-Guerrero system \eqref{CG} is the presence of an exponentially small eigenvalue. This eigenvalue is a consequence of the translation invariance of the operator. Indeed, the derivative of the shock $\de_x(U^\eps_\sigma)$, corresponding to a shift of the location of the shock, lies within the kernel of the operator $\Le_\sigma$, with exponentially small boundary values. We refer to \cite{beck2009using} for more details regarding the metastability.\newline\newline
The results from \cites{CG,glass2010complex} make use of the strong dissipation of the system. If we wish to use a similar approach, we need to treat the term corresponding to the first eigenvalue separately.\\
Moreover, in order to apply the method of moments, we need sharper estimates regarding the distribution of the eigenvalues, as well as some information about the eigenfunctions. We obtain the following result: 
\begin{Lemma}\label{prop1}
    For $\sigma\in(-L,L)$ and $\eps>0$ small enough, the eigenvalues $(\lambda_{\sigma,k}^\eps)_{k\ge 0}$ of the operator $\Le_\sigma$ acting on $H_0^1\cap H^2(-L,L)$ verify the following estimates: \begin{equation}\label{dist0}\exists C>0 : 0<\lambda_{\sigma,0}^\eps<\frac{C}{\eps}
e^{-\frac{L}{2\eps}},\end{equation}
    \begin{equation}\label{dist}\frac1{4\eps}+k^2\frac{\pi^2\eps}{4L^2}<\lambda_{\sigma,k}^\eps<\frac1{4\eps}+(k+1)^2\frac{\pi^2\eps}{4L^2}, \quad k\ge 1\end{equation}\begin{equation}\label{gap}|\lambda_{\sigma,k}^\eps-\lambda_{\sigma,j}^\eps|\ge |k^2-j^2|\frac{\eps\pi^2}{4L^2}, \quad j,k\ge 1.\end{equation}
    Moreover, the associated eigenfunctions 
    $(\psi_{\sigma,k}^\eps)_{k\ge 0}$ of $(\Le_\sigma)^*$ are such that \begin{align}
        \frac{\|\psi_{\sigma,0}^\eps\|_{L^2(-L,L)}}{|\eps(\psi_{\sigma,0}^\eps)'(-L)|}&\le 2\sqrt{2L},\label{0}\\
        \frac{\|\psi_{\sigma,k}^\eps\|_{L^2(-L,L)}}{|\eps(\psi_{\sigma,k}^\eps)'(-L)|}&\le \begin{cases}\frac{4L}{k\pi\sqrt{\eps}},&\sigma\ge0\\\frac{4L}{k\pi\sqrt\eps}e^{\frac{2|\sigma|}\eps},&\sigma<0\end{cases},\quad k\ge1.\label{k}
    \end{align}
    More precisely, the first eigenfunction satisfies, for some constant $c>0$ independent of the parameters, and for $\eps$ small enough, \begin{equation}\label{lim psi0}
        \left\|\frac{\psi_{\sigma,0}^\eps}{\eps(\psi_{\sigma,0}^\eps)'(-L)}-1\right\|_{L^2(-L,L)}\le c\eps.
    \end{equation}
\end{Lemma}

Above, the operator $(\Le_\sigma)^*$  denotes the adjoint operator of $\Le_\sigma$ with homogeneous Dirichlet condition, namely \[(\mathcal L_\sigma^\eps)^*=-\eps\de_x^2-U_\sigma^\eps \de_x.\] An eigenfunction of $(\Le_\sigma)^*$ is a function $\psi_\sigma^\eps$ that satisfies, for some $\lambda_\sigma^\eps\in\mathbb R$,
the eigenvalue problem \begin{equation}\label{eigenL}
        \left\{\begin{aligned}
            (\Le_\sigma)^*\psi_\sigma^\eps&=\lambda_\sigma^\eps\psi_\sigma^\eps,\\
            \psi_\sigma^\eps(\pm L)&=0,
        \end{aligned}\right.
    \end{equation} 

   \begin{remark}
      The distribution of the eigenvalues \eqref{dist} and the information about the gap \eqref{gap} are important in the computations later, and are not directly obtainable to our knowledge by standard methods such as perturbation tools, Sturm-Liouville theory or min-max formulation. We also insist on the fact that the eigenvalues of our operator $(\Le_\sigma)^*$ strongly resemble the ones of the Coron-Guerrero operator, for $M=\pm1$ and an interval of length $2L$.
  \end{remark} 

\begin{proof} In a first time, we reduce the eigenvalue problem \eqref{eigenL} to an equivalent eigenvalue problem for a Schrödinger operator with a constant potential, by first reducing it to a self-adjoint operator.
Then, we directly derive the solutions of the reduced eigenvalue problem, and deduce explicit expressions for the eigenvalues and eigenfunctions of the operator $(\Le_\sigma)^*$ acting on $H^1_0\cap H^2(-L,L).$\\

\textbf{Step 1. Reduction to a self-adjoint operator.}

\begin{Lemma}\label{lem conj 1}
    Let\[P_\sigma^\eps:=-\eps^2\de_x^2-\frac14+\frac12\tanh^2\left(\frac{x-\sigma}{2\eps}\right)=-\eps^2\de_x^2+\left(\frac{(U^\eps_\sigma)^2}{4}+\eps \frac{\de_x U^\eps_\sigma}2\right).\]
    Then the eigenvalue problem \eqref{eigenL} is equivalent to the problem \begin{equation}\label{eigenP}
        \left\{\begin{aligned}
            \frac{P_\sigma^\eps}\eps \varphi_\sigma^\eps&=\lambda^\eps\varphi_\sigma^\eps,\\
            \varphi_\sigma^\eps(\pm L)&=0.
        \end{aligned}\right.
    \end{equation}
    More precisely, a pair $(\lambda_\sigma^\eps, \psi_\sigma^\eps)$ is solution of \eqref{eigenL} if and only if the pair $(\lambda_\sigma^\eps, \sech\left(\frac x{2\eps}\right) \psi_\sigma^\eps)$ solves \eqref{eigenP}.
\end{Lemma}
\begin{proof}

 Following \cite{laurent2016uniform,folino2017metastability} we compute, for $u\in L^2(-L,L)$:
 \begin{align*}
     \sech\left(\frac{x-\sigma}{2\eps}\right)(-\eps\de_x^2)\cosh\left(\frac{x-\sigma}{2\eps}\right)u&=-\eps\de_x^2u - \tanh\left(\frac{x-\sigma}{2\eps}\right)\de_x u -\frac1{4\eps}u,\\
     \sech\left(\frac{x-\sigma}{2\eps}\right)\left(\tanh\left(\frac{x-\sigma}{2\eps}\right)\de_x\right)\cosh\left(\frac{x-\sigma}{2\eps}\right)u&=\tanh\left(\frac{x-\sigma}{2\eps}\right)\de_xu+\frac1{2\eps}\tanh^2\left(\frac{x-\sigma}{2\eps}\right)u.
\end{align*}
Summing those two equalities give the conjugation relation:
    \begin{equation}\label{conjug}\frac{P_\sigma^\eps}\eps=\sech\left(\frac{x-\sigma}{2\eps}\right)\left(\mathcal L_\sigma^\eps\right)^*\cosh\left(\frac{x-\sigma}{2\eps}\right).\end{equation}  
    The equivalence of the eigenvalues problems \eqref{eigenL} and \eqref{eigenP} directly follows from this conjugation.\end{proof}

     \textbf{Step 2. Reduction to a constant potential.}
    \begin{Lemma}\label{lem conj 2}
        Let \[a_\sigma^\eps:=\eps\de_x+\frac12\tanh\left(\frac{x-\sigma}{2\eps}\right)=\eps\de_x-\frac12U^\eps_\sigma.\]
        The eigenvalue problem \eqref{eigenP} reduces to the constant-potential problem \begin{equation}\label{aa*}
        \left\{\begin{aligned}
            \left(-\eps^2\de_x^2+\frac14\right)f_\sigma^\eps&=\eps\lambda_\sigma^\eps f_\sigma^\eps,\\
            (a_\sigma^\eps)^*f_\sigma^\eps(\pm L)&=0.
        \end{aligned}\right.
    \end{equation}
    More precisely, if an eigenpair $(\lambda_\sigma^\eps, \varphi_\sigma^\eps)$ solves the problem \eqref{eigenP} with $a^\eps_\sigma \varphi_\sigma^\eps\ne 0$, then $(\lambda_\sigma^\eps, a^\eps_\sigma \varphi_\sigma^\eps)$ is an eigenpair of the constant-potential problem \eqref{aa*}.\\
    Conversely, if $(\lambda_\sigma^\eps, f^\eps_\sigma)$ is an eigenpair of \eqref{aa*} with $(a_\sigma^\eps)^* f_\sigma^\eps\ne0$, then $(\lambda_\sigma^\eps, (a_\sigma^\eps)^* f_\sigma^\eps)$ is an eigenpair of the problem \eqref{eigenP}.
    \end{Lemma}
    \begin{proof}
    As seen in \cite[Proposition 3.1]{carbou2010stability}, the operator $P_\sigma^\eps$ can be factored under the form \[P_\sigma^\eps=(a_\sigma^\eps)^* a_\sigma^\eps.\]
    Moreover, switching $a_\sigma^\eps$ and its conjugate gives \[a_\sigma^\eps (a_\sigma^\eps)^*=-\eps^2\de_x^2+\frac{(U^\eps_\sigma)^2}{4}-\eps\frac{\de_x U^\eps_\sigma}{2}.\]
    Note that \[\de_x\left(\frac{(U^\eps_\sigma)^2}{4}+\eps\frac{\de_x U^\eps_\sigma}{2}\right)=\frac12\left( U_\sigma^\eps\de_x U_\sigma^\eps - \eps\de_x^2 U_\sigma^\eps\right)=0,\] as $U^\eps_\sigma$ is a stationary solution of the Burgers' equation. This transformation therefore will not lead to a constant potential for other conservation laws.\\
    We then have the much simpler operator \[a_\sigma^\eps(a_\sigma^\eps)^*=-\eps^2\de_x^2+\frac14.\]
    Now let $(\lambda_\sigma^\eps,\varphi_\sigma^\eps)$ a pair solving the problem \eqref{eigenP}, and take $f_\sigma^\eps:= a_\sigma^\eps \varphi_\sigma^\eps.$\\
    It holds \[(a_\sigma^\eps)^*a_\sigma^\eps \varphi_\sigma^\eps = \eps\lambda_\sigma^\eps\varphi_\sigma^\eps.\]
    Applying the operator $a_\sigma^\eps$ to this equality then yields \[a_\sigma^\eps(a_\sigma^\eps)^* f_\sigma^\eps=\eps\lambda_\sigma^\eps f_\sigma^\eps.\]
    For the boundary conditions, we note that \[(a_\sigma^\eps)^*f_\sigma^\eps=P_\sigma^\eps\varphi_\sigma^\eps=\eps\lambda_\sigma^\eps\varphi_\sigma^\eps,\]
    so that it holds \[(a_\sigma^\eps)^*f_\sigma^\eps(\pm L)=0.\]
    The pair $(\lambda_\sigma^\eps, f_\sigma^\eps)$ therefore solves the eigenvalue problem \eqref{aa*}. The converse behaves the exact same way.\end{proof}
   
      \textbf{Step 3. An exponentially small eigenvalue.}

    We now resume the proof of Lemma \ref{prop1} by studying the simple eigenvalue problem \eqref{aa*}. In a first time, we focus on proving the estimates on the first eigenvalue \eqref{dist0}, as well as the properties of the associated eigenfunction \eqref{0}, \eqref{lim psi0}.\newline\newline
    If an eigenvalue satisfies $\lambda_\sigma^\eps<\frac1{4\eps}$, then there exists $A,B\in\R$ such that the associated eigenfunction $f_\sigma^\eps$ is of the form \begin{equation}\label{eigen 0 f}f_\sigma^\eps(x)=A\exp\left(\frac{\sqrt{\frac14-\eps\lambda_\sigma^\eps}}{\eps}x\right)+B\exp\left(-\frac{\sqrt{\frac14-\eps\lambda_\sigma^\eps}}{\eps}x\right)\end{equation}
    The boundary condition $(a_\sigma^\eps)^*f_\sigma^\eps(\pm L)=0$ then translates to \begin{align}\notag A\exp\left(-\frac{\sqrt{\frac14-\eps\lambda_\sigma^\eps}L}{\eps}\right)\left(-\sqrt{\frac14-\eps\lambda_\sigma^\eps}-\frac12\tanh\left(\frac{L+\sigma}{2\eps}\right)\right)\\+B\exp\left(\frac{\sqrt{\frac14-\eps\lambda_\sigma^\eps}L}{\eps}\right)\left(\sqrt{\frac14-\eps\lambda_\sigma^\eps}-\frac12\tanh\left(\frac{L+\sigma}{2\eps}\right)\right)=0,\\\notag\\A\exp\left(\frac{\sqrt{\frac14-\eps\lambda_\sigma^\eps}L}{\eps}\right)\left(-\sqrt{\frac14-\eps\lambda_\sigma^\eps}+\frac12\tanh\left(\frac{L-\sigma}{2\eps}\right)\right)\\+B\exp\left(\frac{\sqrt{\frac14-\eps\lambda_\sigma^\eps}L}{\eps}\right)\left(\sqrt{\frac14-\eps\lambda_\sigma^\eps}+\frac12\tanh\left(\frac{L-\sigma}{2\eps}\right)\right)=0.\end{align}
    Introducing \[\mu_\sigma^\eps:=\sqrt{\frac14-\eps\lambda_\sigma^\eps},\]
    this system admits a nontrivial solution $(A,B)$ if and only if\begin{align}\notag
        \left(\mu_\sigma^\eps-\frac12\tanh\left(\frac{L-\sigma}{2\eps}\right)\right)\left(\mu_\sigma^\eps-\frac12\tanh\left(\frac{L+\sigma}{2\eps}\right)\right)=\\\left(\mu_\sigma^\eps+\frac12\tanh\left(\frac{L-\sigma}{2\eps}\right)\right)\left(\mu_\sigma^\eps+\frac12\tanh\left(\frac{L+\sigma}{2\eps}\right)\right)\exp\left(-\frac{4\mu_\sigma^\eps L}\eps\right).\label{cond mu}
    \end{align}
    The values $\mu_\sigma^\eps=0$ and $\mu_\sigma^\eps=\frac12$ are trivial solutions, but we discard them as they lead to $f_\sigma^\eps=0$ or $f_\sigma^\eps\in\ker((a_\sigma^\eps)^*)$.\\
    On the interval $[0,\frac14)$, multiplying the above condition by $\exp(2\mu_\sigma^\eps L/\eps)$ gives, for $\eps$ small enough, a LHS which is increasing, and a RHS which is decreasing. Therefore there is no solution $\mu_\sigma^\eps$ of \eqref{cond mu} in $(0,\frac14).$\\
    On $(\frac14,\frac12]$, the function given by the difference LHS-RHS is convex, and thus admits at most two zeroes. By the intermediate value theorem, it has exactly two zeroes, which are each located between a root of \[\mu\mapsto \left(\mu-\frac12\tanh\left(\frac{L-\delta}{2\eps}\right)\right)\left(\mu-\frac12\tanh\left(\frac{L+\delta}{2\eps}\right)\right),\]and a root of \[\mu\mapsto \left(\mu-\frac12\tanh\left(\frac{L-\delta}{2\eps}\right)\right)\left(\mu-\frac12\tanh\left(\frac{L+\delta}{2\eps}\right)\right)-\exp\left(\frac{2L}\eps\right).\]
    To conclude, the condition \eqref{cond mu} admits exactly three solutions in $[0,\frac12]$: $0,\frac12$ and some $\mu_{\sigma,0}^\eps$ verifying \begin{equation}\label{est mu}\frac12-\mu_{\sigma,0}^\eps\sim \exp\left(-\frac{L-|\sigma|}{\eps}\right).\end{equation}
    The associated eigenvalue then verifies\begin{equation}
        \lambda_{\sigma,0}^\eps\sim\frac1\eps\cdot\exp\left(-\frac{L-|\delta|}\eps\right),
    \end{equation}
    and any nontrivial pair of coefficients $(A,B)$ satisfies \[\frac AB\sim -\exp\left(\frac{\sigma}\eps\right).\]
    From Lemmas \ref{lem conj 1}, \ref{lem conj 2}, the  corresponding eigenfunction for $(\Le)^*$ is $\psi_0^\eps(x)=\cosh\left(\frac{x}{2\eps}\right)\varphi_0^\eps(x),$ where $f_{\sigma,0}^\eps$ is given by \eqref{eigen 0 f}. Direct calculations then give \begin{align}
        \psi_{\sigma,0}^\eps(x)=A\left[(-\mu_{\sigma,0}^\eps+\frac12)\exp\left((\mu_{\sigma,0}^\eps+\frac12)\frac x\eps-\frac{\sigma}{2\eps}\right)+(-\mu_{\sigma,0}^\eps-\frac12)\exp\left((\mu_{\sigma,0}^\eps-\frac12)\frac x\eps+\frac{\sigma}{2\eps}\right)\right]+\notag\\B\left[(\mu_{\sigma,0}^\eps+\frac12)\exp\left((-\mu_{\sigma,0}^\eps+\frac12)\frac x\eps-\frac{\sigma}{2\eps}\right)+(\mu_{\sigma,0}^\eps-\frac12)\exp\left((-\mu_{\sigma,0}^\eps-\frac12)\frac x\eps+\frac{\sigma}{2\eps}\right)\right].\label{defphi0}
    \end{align}
    We set $A=1$ so that $B\sim -\exp(-\sigma/\eps)$. Using this as well as the asymptotic behavior \eqref{est mu} gives us the estimates, as $\eps$ goes to $0$:\begin{align}
        \eps(\psi_{\sigma,0}^\eps)'(-L)&\sim -\exp\left(\frac{|\sigma|}\eps-\frac{\sigma}{2\eps}\right)\label{up0},\\
        \|\psi_{\sigma,0}^\eps\|_{L^2(-L,L)}&\sim\sqrt{2L}\cdot\left(\exp\left(\frac\sigma{2\eps}\right)+\exp\left(-\frac{3\sigma}{2\eps}\right)\right)\sim\sqrt{2L}\exp\left(\frac{|\sigma|}\eps-\frac\sigma{2\eps}\right).\label{low0}
    \end{align}
    Combining the bounds \eqref{up0} and \eqref{low0} then gives the desired estimates \eqref{0}, for $\eps$ small enough.\\
    We may refine the analysis as $\eps$ goes to $0$, and obtain the limit behavior\begin{align}
        &\eps(\psi_{\sigma,0}^\eps)'(-L)=-\exp\left(\frac{|\sigma|}{\eps}-\frac{\sigma}{2\eps}\right)(1+O(e^{-\frac c\eps})),\\
        &\left\|-\exp\left(\frac{|\sigma|}\eps-\frac{\sigma}{2\eps}\right)\psi_{\sigma,0}^\eps-1\right\|_{L^2(-L,L)}\le 4\eps.
    \end{align}
    This finally yields \eqref{lim psi0}.
    
    \textbf{Step 4. The other eigenvalues.}
    We now conclude the proof of Lemma \ref{prop1}, by studying the rest of the spectrum and showing that it follows the distribution \eqref{dist}. We also study the associated eigenfunctions and prove they verify the boundary behavior \eqref{k}.\newline\newline
    Now, let us consider eigenvalues verifiying $\lambda_\sigma^\eps\ge\frac1{4\eps}$. Then any associated eigenfunction can be written as \[f_\sigma^\eps(x)=A\cos\left(\frac{\sqrt{\eps\lambda_\sigma^\eps-\frac14}}{\eps}(x-\sigma)\right)+B\sin\left(\frac{\sqrt{\eps\lambda_\sigma^\eps-\frac14}}{\eps}(x-\sigma)\right).\]
    The boundary conditions then translate to:
    \begin{align}
        \notag A\left(-\sqrt{\eps\lambda_\sigma^\eps-\frac14}\sin\left(\frac{\sqrt{\eps\lambda_\sigma^\eps-\frac14}}{\eps}(L+\sigma)\right)-\frac12\tanh\left(\frac{L+\sigma}{2\eps}\right)\cos\left(\frac{\sqrt{\eps\lambda_\sigma^\eps-\frac14}}{\eps}(L+\sigma)\right)\right)+\\
        B\left(-\sqrt{\eps\lambda_\sigma^\eps-\frac14}\cos\left(\frac{\sqrt{\eps\lambda_\sigma^\eps-\frac14}}{\eps}(L+\sigma)\right)+\frac12\tanh\left(\frac{L+\sigma}{2\eps}\right)\sin\left(\frac{\sqrt{\eps\lambda_\sigma^\eps-\frac14}}{\eps}(L+\sigma)\right)\right)=0,\label{left}\\\notag\\
        \notag A\left(\sqrt{\eps\lambda_\sigma^\eps-\frac14}\sin\left(\frac{\sqrt{\eps\lambda_\sigma^\eps-\frac14}}{\eps}(L-\sigma)\right)+\frac12\tanh\left(\frac{L-\sigma}{2\eps}\right)\cos\left(\frac{\sqrt{\eps\lambda_\sigma^\eps-\frac14}}{\eps}(L-\sigma)\right)\right)+\\
        B\left(-\sqrt{\eps\lambda_\sigma^\eps-\frac14}\cos\left(\frac{\sqrt{\eps\lambda_\sigma^\eps-\frac14}}{\eps}(L-\sigma)\right)+\frac12\tanh\left(\frac{L-\sigma}{2\eps}\right)\sin\left(\frac{\sqrt{\eps\lambda_\sigma^\eps-\frac14}}{\eps}(L-\sigma)\right)\right)=0.\label{right}
    \end{align}
    We introduce the two unknowns \[\theta_{\sigma,\pm}^\eps:=\arctan\left(\frac{\sqrt{4\eps\lambda_\sigma^\eps-1}}{\tanh\left(\frac {L\mp\sigma}{2\eps}\right)}\right).\]
    The constraints \eqref{left} and \eqref{right} may then be rewritten as
    \begin{align}
        -A\cos\left(\frac{\sqrt{\eps\lambda_\sigma^\eps-\frac14}}{\eps}(L+\sigma)-\theta^\eps_{\sigma,-}\right)+B\sin\left(\frac{\sqrt{\eps\lambda_\sigma^\eps-\frac14}}{\eps}(L+\sigma)-\theta^\eps_{\sigma,-}\right)=0,\label{left2}\\
        A\cos\left(\frac{\sqrt{\eps\lambda_\sigma^\eps-\frac14}}{\eps}(L-\sigma)-\theta^\eps_{\sigma,+}\right)+B\sin\left(\frac{\sqrt{\eps\lambda_\sigma^\eps-\frac14}}{\eps}(L-\sigma)-\theta^\eps_{\sigma,+}\right)=0,\label{right2}
    \end{align}
    A non-trivial solution exists if and only if  \begin{align*}\sin\left(\frac{\sqrt{\eps\lambda_\sigma^\eps-\frac14}}{\eps}(L+\sigma)-\theta^\eps_{\sigma,-}+\frac{\sqrt{\eps\lambda_\sigma^\eps-\frac14}}{\eps}(L-\sigma)-\theta^\eps_{\sigma,+}\right)=0.\end{align*}
    Therefore $\lambda^\eps>\frac1{4\eps}$ is an eigenvalue if and only if it satisfies
    \begin{equation}\label{const}
        \sqrt{\eps\lambda_\sigma^\eps-\frac14}=\frac{\eps}{2L}(k\pi+\theta_{\sigma,-}^\eps+\theta_{\sigma,+}^\eps),\quad \text{for some }k\ge 1.
    \end{equation}
    For $\eps$ small enough, the function \[x\mapsto x-\frac{\eps}{2L}\left(\arctan\left(\frac{2x}{\tanh\left(\frac {L+\sigma}{2\eps}\right)}\right)+\arctan\left(\frac{2x}{\tanh\left(\frac {L-\sigma}{2\eps}\right)}\right)\right)\]s increasing on $\R^+$, so that the constraint \eqref{const} has a unique solution $\lambda_k^\eps$ for each $k\ge1$. Moreover, since $\theta_{\sigma,\pm}^\eps\in \left(0,\frac\pi2\right)$ and $\theta^\eps_{\sigma,\pm}$ are increasing functions of $\lambda_\sigma^\eps$, we immediately deduce the properties \eqref{dist} and \eqref{gap}.\\
    An associated eigenfunction solution of the problem \eqref{aa*} is then given, for $k\ge 1$, by \[f_{\sigma,k}^\eps(x)=\sin\left(\frac{\sqrt{\eps\lambda_{\sigma,k}^\eps-\frac14}}{\eps}(x+L) - \theta_{\sigma,-,k}^\eps\right).\]
    From Lemma \ref{lem conj 2}, an associated eigenfunction of the self-adjoint problem \eqref{eigenP} is \begin{align}\label{defphik}\varphi_{\sigma,k}(x)=\notag&-\sqrt{\eps\lambda_{\sigma,k}^\eps-\frac14}\cos\left(\frac{\sqrt{\eps\lambda_{\sigma,k}^\eps-\frac14}}{\eps}(x+L-\theta_{\sigma,-,k}^\eps)\right)\\&+\frac12\tanh\left(\frac{x-\sigma}{2\eps}\right)\sin\left(\frac{\sqrt{\eps\lambda_{\sigma,k}^\eps-\frac14}}{\eps}(x+L)-\theta_{\sigma,-,k}^\eps\right).\end{align}
    From Lemma \ref{lem conj 1}, the corresponding eigenfunction for $(\Le_\sigma)^*$ is then obtained by taking $\psi_{\sigma,k}^\eps(x)=\cosh\left(\frac{x-\sigma}{2\eps}\right)\varphi_{\sigma,k}^\eps(x)$. Direct estimates then give the upper bound \begin{equation}\label{up}\|\psi_k^\eps\|_{L^2(-L,L)}\le \eps\sqrt{\lambda_k^\eps}\cosh\left(\frac{L+|\sigma|}{2\eps}\right),\end{equation} as well as the lower bound on the derivative \begin{equation}\label{der}|(\psi_k^\eps)'(-L)| \ge \frac12\sqrt{\lambda_k^\eps}\sqrt{\lambda_k^\eps-\frac1{4\eps}} \cosh\left(\frac{L+\sigma}{2\eps}\right).\end{equation}
    Combining \eqref{up} and \eqref{der} leads to the desired estimate \eqref{k}.
 
\end{proof}

\section{Proof of the upper bound on $T_{\rm{unif}}$}

We fix here and in the rest of the paper the parameters $L>0$ and $\sigma\in (-L,L).$

\subsection{Scheme of the proof}\label{sec-upper}
As mentioned in Section 2, we control the solution of system \eqref{cont} to zero in two steps, splitting the imparted time interval $(0,T)$ into two parts $(0,\tau)$ and 
$(\tau,T)$, for an intermediate time $\tau \in (0,T)$ which will be determined later, see \eqref{cla}.
In a first time, we eliminate the first mode of $\ue$, namely its projection on the eigenfunction $\psi_{\sigma,0}^\eps$. Then, we control the rest by exploiting the strong dissipation, and concluding with the method of moments, similarly to the work of O. Glass in \cite{glass2010complex} and P. Lissy in \cite{lissy2012link}. This is the object of the two following preliminary lemmas. 
Recall that the functions $(\psi_{\sigma,k}^\eps)_k$ are the eigenfunctions of the operator $(\Le_\sigma)^*$
    and are defined in
    \eqref{defphi0} and below 
    \eqref{defphik}. 

\begin{Lemma}\label{lem1}
Let $T>0$ and  $\tau \in (0,T)$. 
  For any $\eps>0$ and any initial datum $u_0\in L^2(-L,L)$, there is a control function $h_{1}^\eps\in L^2(0,\tau)$ verifying: 
    \begin{align}
        \|h_{1}^\eps\|_{L^2(0,\tau)}\le\frac{2\sqrt{2L}}{\sqrt\tau} \|u_0\|_{L^2(-L,L)}, \label{cout1}
    \end{align}
   such that the weak solution $\ue$ of \eqref{cont asym} satisfies 
   at time $\tau$, 
   for any $k \in \N$,
    \begin{equation}\label{ytau}
        \langle \ue(\tau),\psi_{\sigma,k}^\eps\rangle=e^{-\lambda_{\sigma,k}^\eps\tau}\langle u_0,\psi_{\sigma,k}^\eps\rangle - \frac{(\psi_{\sigma,k}^\eps)'(-L)}{(\psi_{\sigma,0}^\eps)'(-L)}\cdot\frac{e^{-\lambda_{\sigma,0}^\eps\tau}\langle u_0,\psi_{\sigma,0}^\eps\rangle}{\tau} \int_0^\tau e^{-(\lambda_{\sigma,k}^\eps-\lambda_{\sigma,0}^\eps)t}\,\dd t, 
    \end{equation}
    where the brackets $\langle\cdot,\cdot\rangle$ denote the inner product in $L^2(-L,L)$.\\
    In particular, for $k=0$, 
    \begin{equation}   \langle \ue(\tau),\psi_{\sigma,0}^\eps\rangle=0.\label{mode1}
        \end{equation}
\end{Lemma}

In the following lemma we continue from the state $\ue(\tau)$ as a new initial data for the system \eqref{cont} now considered on 
 the time interval $(\tau,T)$.

\begin{Lemma}\label{lem2}
Let $T>0$ and  $\tau \in (0,T)$ with 
\begin{equation} \label{ratio}
  T-\tau>T^*,
\end{equation}
where the minimal time $T^*$ is given in \eqref{def t*}.\\
Then there exists $C>0$ such that for any $u_0\in L^2(-L,L)$, for any $\eps>0$ , 
there exists a control function $h_{2}^\eps\in L^2(\tau,T)$ with 
    \begin{equation}\label{cout2}\|h_{2}^\eps\|_{L^2(\tau,T)}\le Ce^{-\frac C\eps}\left\|u_0\right\|_{L^2(-L,L)},\end{equation}
    such that 
    the weak solution $\ue$ of \eqref{cont} on the time interval $(\tau,T)$ starting at time $\tau $ with the initial data $\ue(\tau)$ given by the formula  \eqref{ytau},  for  $k \in \N$,
    satisfies 
    $\ue(T)=0$.
\end{Lemma}
The proofs of Lemma \ref{lem1} and \ref{lem2} are respectively given in the next two  sections. Let us take them as granted for the moment and conclude the proof of the upper bound on $T_{\rm{unif}}$, that is $T_{\rm{unif}}\leq T^*$ as well as the upper bound of the cost
\eqref{major}. 
We fix a time $T>T^*$, and set 
\begin{equation} \label{cla}
    \tau=\frac{T-T^*}2,
\end{equation} so that $T-\tau>T^*$.
Let $\eps>0$ and $u_0\in L^2(-L,L)$. We obtain our control $h^\eps\in L^2(0,T)$ by concatenating the two controls $h_1^\eps, h_2^\eps$ built in Lemmas \ref{lem1} and \ref{lem2}, namely
\[h^\eps(t):=\begin{cases}
    h_1^\eps(t)&\text{if } 0 < t<\tau,\\
    h_2^\eps(t)&\text{if }\tau <  t < T.
\end{cases}\]
This choice of control steers the initial datum $u_0$ to the final state $\ue(T)=0$.\\ 
Moreover such a control $h^\eps$ verifies \begin{align}
    \notag\|h^\eps\|_{L^2(0,T)}&\le \|h_1^\eps\|_{L^2(0,\tau)}+\|h_2^\eps\|_{L^2(\tau,T)}\\
    &\le \left(\frac{4\sqrt{L}}{\sqrt{T-T^*}}+Ce^{-\frac C\eps}\right)\left\|u_0\right\|_{L^2(-L,L)},\label{48}
\end{align}
which proves the uniform controllability time verifies $T_{\rm{unif}}\le T^*$, and leads to \eqref{major}.

\subsection{Controllability of the first mode. Proof of Lemma \ref{lem1}}
\label{sec-lem1}

We begin with a preliminary lemma exploiting duality equalities to describe the behavior of each mode of the solution.

\begin{Lemma}\label{lemdual}
    Let $\eps>0$, $t_1, t_2\in[0,T]$, and $k\ge0$. The evolution of the $k$-th mode of the solution $\ue$ of the control problem \eqref{cont asym} is prescribed by the relation:
    \begin{equation}
        \langle \ue(t_2),\psi_{\sigma,k}^\eps\rangle = e^{-\lambda_{\sigma,k}^\eps(t_2-t_1)}\langle \ue(t_1),\psi_{\sigma,k}^\eps\rangle+\eps(\psi_{\sigma,k}^\eps)'(-L)\int_{t_1}^{t_2}e^{-\lambda_{\sigma,k}^\eps (t-t_1)}h(t_1+t_2-t)\,\dd t.
    \end{equation}
\end{Lemma}
\begin{proof}
    We introduce $\zeta^\eps_{\sigma,k}(t,x)$ as the solution to the adjoint uncontrolled system \begin{equation}
        \left\{\begin{aligned}
            \de_t\zeta^\eps_{\sigma,k}+(\Le)^*\zeta^\eps_{\sigma,k}&=0,\\
            \zeta^\eps_{\sigma,k}(t,-L)=\zeta^\eps_{\sigma,k}(t,L)&=0,\\
            \zeta^\eps_{\sigma,k}(t_1,x)=\psi_{\sigma,k}^\eps(x).
        \end{aligned}\right.
    \end{equation}
    Using the definition of the operators $\Le_\sigma, (\Le_\sigma)^*$ and integrating by parts, it holds \begin{equation}\frac{\dd}{\dd t}\langle \ue(t),\zeta^\eps_{\sigma,k}(t_1+t_2-t)\rangle=\eps \de_x\zeta^\eps_{\sigma,k}(t_1+t_2-t,-L)\ue(t,-L).\end{equation}Therefore, integrating this relation from $t_1$ to $t_2$, we have the duality equality \begin{equation}\label{dual}\langle \ue(t_2),\zeta^\eps_{\sigma,k}(t_1)\rangle=\langle \ue(t_1), \zeta^\eps_{\sigma,k}(t_2)\rangle + \int_{t_1}^{t_2}\eps \de_x\zeta^\eps_{\sigma,k}(t,-L) h(t_1+t_2-t)\,\dd t.\end{equation}
    Then, we recall $\psi_{\sigma,k}^\eps$ is an eigenfunction of $(\Le_\sigma)$, so that $\zeta^\eps_{\sigma,k}$ is explicitly known as \[\zeta^\eps_{\sigma,k}(t,x)=\psi_{\sigma,k}^\eps(x)e^{-\lambda_{\sigma,k}^\eps (t-t_1)}.\]Substituting this into \eqref{dual} yields \[\langle \ue(t_2),\psi_{\sigma,k}^\eps\rangle=e^{-\lambda_{\sigma,k}^\eps(t_2-t_1)}\langle \ue(t_1),\psi_{\sigma,k}^\eps\rangle + \eps(\psi_k^\eps)'(-L)\int_{t_1}^{t_2} e^{-\lambda_{\sigma,k}^\eps (t-t_1)}h(t_1+t_2-t)\,\dd t,\] which concludes the proof.\end{proof}

    Applying Lemma \ref{lemdual} with $k=0$, $t_1=0, t_2=\tau$ gives the relation \begin{equation}
        \langle \ue(\tau),\psi_{\sigma,0}^\eps\rangle= e^{-\lambda_{\sigma,0}^\eps\tau}\langle u_0,\psi_{\sigma,0}^\eps\rangle +\eps(\psi_{\sigma,0}^\eps)'(-L)\int_0^\tau e^{-\lambda_{\sigma,0}^\eps t}h(\tau-t)\,\dd t.
    \end{equation}
    We then choose, for $t\in[0,\tau],$ \begin{equation}\label{h}h^\eps_1(t)=- \frac{\langle u_0,\psi_{\sigma,0}^\eps\rangle}{\eps(\psi_{\sigma,0}^\eps)'(-L)} \cdot\frac{e^{-\lambda_{\sigma,0}^\eps t}}{\tau}.\end{equation}It immediately follows that \[\langle \ue(\tau),\psi^\eps_{\sigma,0}\rangle=0.\]
    We have the upper bound \[|h_1^\eps(t)|\le \frac{\|\psi_{\sigma,0}^\eps\|_{L^2(-L,L)}}{|\eps(\psi_0^\eps)'(-L)|} \cdot\frac{\|u_0\|_{L^2(-L,L)}}{\tau}.\]Using the estimate \eqref{0} then immediately yields \[\|h^\eps_1\|_{L^2(0,\tau)}\le \frac{2\sqrt{2L}}{\sqrt\tau} \|u_0\|_{L^2(-L,L)}.\]
    Finally, to perform the desired identities on $\ue(\tau)$, we apply Lemma \ref{lemdual} again, for $k\ge1$ and $t_1=0,t_2=\tau$. This leads to the duality equalities \[\langle \ue(\tau),\psi^\eps_{\sigma,k}\rangle=e^{-\lambda_{\sigma,k}\tau}\langle u_0,\psi^\eps_{\sigma,k}\rangle + \eps(\psi_{\sigma,k}^\eps)'(-L)\int_0^\tau e^{-\lambda_{\sigma,k}^\eps t}h_1^\eps(\tau-t)\,\dd t.\]
    Plugging the expression of the control $h_1^\eps(t)$, the solution at final time then verifies, for $k\ge1$, \begin{equation}
        \langle \ue(\tau),\psi^\eps_{\sigma,k}\rangle=e^{-\lambda_{\sigma,k}^\eps\tau}\langle u_0,\psi_{\sigma,k}^\eps\rangle - \frac{(\psi_{\sigma,k}^\eps)'(-L)}{(\psi_{\sigma,0}^\eps)'(-L)}\cdot\frac{e^{-\lambda_{\sigma,0}^\eps\tau}\langle u_0,\psi_{\sigma,0}^\eps\rangle}{\tau} \int_0^\tau e^{-(\lambda_{\sigma,k}^\eps-\lambda_{\sigma,0}^\eps)t}\,\dd t.
    \end{equation}
    This concludes the proof of Lemma \ref{lem1}.\\
    We additionally note that, as $\eps\to0$, we may compute the limit of the control $h_1^\eps$ using the property \eqref{lim psi0} of the first eigenfunction $\psi_{\sigma,0}^\eps$. It holds \begin{align*}
        &\left\|e^{-\lambda_{\sigma,0}^\eps t}-1\right\|_{L^2(0,\tau)}=O(e^{-\frac c\eps}),\\
        &\left\|\frac{\psi_{\sigma,0}^\eps}{\eps(\psi_{\sigma,0}^k)'(-L)}-1\right\|_{L^2(-L,L)}\le c\eps.
    \end{align*}
    Plugging those limits in the expression of the control \eqref{h}, we obtain \begin{equation}\label{71}
        \left\|h_1^\eps+\frac{\int_{-L}^L u_0(x)\,\dd x}{\tau}\right\|_{L^2(0,\tau)}\le c\frac{\eps\|u_0\|_{L^2(-L,L)}}{\sqrt\tau},
    \end{equation}
    which proves the first half of \eqref{7}.

\subsection{Controllability of the other modes. Proof of Lemma \ref{lem2}}
\label{sec-lem2}
Let $T>0$,  $\tau \in (0,T)$ with 
\begin{equation} \label{ratio2}
  T-\tau>T^*.  
\end{equation}
  Let $u_0\in L^2(-L,L)$, and $\ue(\tau)$ prescribed by the conditions \eqref{ytau}.

  We perform the change of variable 
  \begin{equation} \label{mardi}
  \wt\ue(t,x):=\ue(t+\tau,x), \wt{h^\eps}(t):=h^\eps(t+\tau),\quad \wh{T}:=T-\tau,
    \end{equation}
     so that 
    the solution $\ue$ of \eqref{cont asym} on the time interval $(\tau,T)$ starting at time $\tau $ with the initial data $\ue(\tau)$ provides a solution $\wt\ue$ 
    to the system 
    \begin{equation}
        \left\{\begin{aligned}\label{contdiss}
            \de_t \wt\ue+\Le\wt \ue&=0,\\
            \wt\ue(t,-L)&=\wt{h^\eps}(t),\\
            \wt\ue(t,L)&=0,\\
            \wt\ue(0,x)&=\ue(\tau,x),
        \end{aligned}\right.
    \end{equation}
    on $(0,\wh{T})$. We wish to build a control $\wt{h^\eps}\in L^2(0,\wh T)$ such that $\wt\ue(\wh T)=0.$ From Lemma \ref{lemdual}, this is equivalent to finding $\wt{h^\eps}$ such that:
    \begin{equation}\label{moments}
        \forall k\ge0, \quad \int_0^{\wh T}\wt{h^\eps}(t)e^{\lambda_{\sigma,k}^\eps t}\,\dd t=-\frac{\langle\ue(\tau),\psi_{\sigma,k}^\eps\rangle}{\eps(\psi_{\sigma,k}^\eps)'(-L)}.
    \end{equation}
    In particular the control must satisfy \begin{equation*}
        \int_0^{\wh T} \wt{h^\eps}(t)e^{\lambda_{\sigma,0}^\eps t}\,\dd t=0.
    \end{equation*}

Without control the system already dissipates and drives the system close to zero in finite time. However, driving the state exactly to zero in finite time will require a control which may be huge a priori. \\

A classical method used for such moment problems, in particular for controllability results on the heat equation and other parabolic systems, is to look for a control $\wt{h^\eps}$ as a series of functions $(q_{\sigma,k}^\eps)_{k\ge 1}\subset L^2(0,\wh T)$ which are bi-orthogonal to the exponentials $\left(e^{\lambda_{\sigma,j}^\eps t} \right)_{j\ge 0}$ in $L^2(0,\wh T)$, as performed in \cites{glass2010complex, lissy2012link}. To determine such functions, we use the following result.
\begin{Lemma}\label{lem3}
    Let $\wh T>0, \eps>0$. There exists a family $(q_{\sigma,k}^\eps)_{k\ge 1}$ in $L^2(0,\wh T)$ such that:\begin{equation}\label{bior0}
        \forall j\ge0, k\ge 1, \quad \int_0^{\wh T} q_{\sigma,k}^\eps(t)e^{\lambda_{\sigma,j}^\eps t}=\delta_{j,k}.
    \end{equation}
    Moreover, for any $\kappa>1$, there exists such a family and a constant $c=c(\eps,\kappa,L,\sigma,\wt T)$ which is a rational function of its arguments, independent of $k\ge1$, such that\begin{equation}
        \|q_{\sigma,k}^\eps\|_{L^2\left(0,\wh T\right)}\le c\frac1{\lambda_{\sigma,k}^\eps-\frac1{4\eps}} \exp\left(-\frac{\wh T}{8\eps} + \frac{6\kappa L^2}{\eps \wh T}\right)\label{bior}.
    \end{equation}
\end{Lemma}
    Note that imposing the bi-orthogonality condition \eqref{bior0} for $j=0$ allows us to make sure we do not re-excite the first eigenmode.\\
    This result is proved in Appendix \ref{app bior} with the choice of parameters $C_1=\frac1{4\eps}, C_2=\frac{\pi^2\eps}{4L^2}$, and follows the construction of \cite{tenenbaum2007new}, adapting it to not re-excite the first eigenmode and incorporating the dissipation term through the trick of \cite{lissy2012link}. Despite the extra condition on $j=0$, we manage to recover the same estimates as those obtained for the Coron-Guerrero problem on an interval of size $2L$, with a transport speed $M=\pm1$.\\

    Let $\kappa>1$ and $(q_{\sigma,k}^\eps)_{k\ge 1}$ a bi-orthogonal family verifying the estimate \eqref{bior}.\\
    We may now construct the control $\wt{h^\eps}\in L^2(0,\wh T)$, as \begin{equation}
        \wt{h^\eps}(t)=
            -\sum_{k\ge1}c_{\sigma,k}^\eps q_{\sigma,k}^\eps(t),
    \end{equation}
    where the coefficients $c_k$ are given by \begin{equation}
        c_{\sigma,k}^\eps:=-\frac{\langle \ue(\tau),\psi_{\sigma,k}^\eps\rangle}{\eps(\psi_{\sigma,k}^\eps)'(-L)}.
    \end{equation}
    We will prove that this control is well-defined, steers the solution of \eqref{contdiss} to $0$, and satisfies suitable $L^2$ estimates.\\
   This choice of control clearly solves the moment problem \eqref{moments}, and this yields, for $k\ge0$, \begin{equation}
        \langle \wt\ue(\wh T), \psi_{\sigma,k}^\eps\rangle=0.
    \end{equation}
    Therefore, provided $\wt{h^\eps}$ is well-defined, since $\ue(T)=\wt\ue(\wh T)$, the control \[h^\eps(t)=\wt{h^\eps}(t-\tau),\] steers $\ue(\tau)$ to $\ue(T)=0$.\\

    We may now estimate the norm of the control $\wt{h^\eps}$. We start by bounding the coefficients $c_{\sigma,k}^\eps$. Reminding that $\wt\ue(0)=\ue(\tau)$ is given by the conditions \eqref{ytau}, we can rewrite \begin{equation}\label{1}
        c_{\sigma,k}^\eps=e^{-\lambda_{\sigma,k}^\eps\tau} \frac{\langle u_0,\psi_{\sigma,k}^\eps\rangle}{\eps(\psi_{\sigma,k}^\eps)'(-L)}-\frac{e^{-\lambda_{\sigma,0}^\eps\tau}}{\tau}\int_0^\tau e^{-(\lambda_{\sigma,k}^\eps-\lambda_{\sigma,0}^\eps)t}\,\dd t \frac{\langle u_0,\psi_{\sigma,0}^\eps\rangle}{\eps(\psi_{\sigma,0}^\eps)'(-L)}
    \end{equation}
    We estimate both terms separately. For the first term, we apply Cauchy-Schwarz' inequality and use that $\lambda_{\sigma,k}^\eps\tau>0$ to obtain:
    \begin{align}
        \left|e^{-\lambda_{\sigma,k}^\eps\tau} \frac{\langle u_0,\psi_{\sigma,k}^\eps\rangle}{\eps(\psi_{\sigma,k}^\eps)'(-L)}\right|&\le \frac{\left\|u_0\right\|_{L^2(-L,L)} \|\psi_{\sigma,k}^\eps\|_{L^2(-L,L)}}{|\eps(\psi_{\sigma,k}^\eps)'(-L)|}.\label{ck1}
    \end{align}
    For the second term, likewise, using $\lambda_{\sigma,k}^\eps-\lambda_{\sigma,0}^\eps>0$, we get:
    \begin{align}
        \left|\frac{e^{-\lambda_{\sigma,0}^\eps\tau}}{\tau}\int_0^\tau e^{-(\lambda_{\sigma,k}^\eps-\lambda_{\sigma,0}^\eps)t}\,\dd t \frac{\langle u_0,\psi_{\sigma,0}^\eps\rangle}{\eps(\psi_{\sigma,0}^\eps)'(-L)}\right|&\le \frac{\left\|u_0\right\|_{L^2(-L,L)}\|\psi_{\sigma,0}^\eps\|_{L^2(-L,L)}}{|\eps(\psi_{\sigma,0}^\eps)'(-L)|}.\label{ck2}
    \end{align}
    Combining the estimates \eqref{ck1} and \eqref{ck2}, as well as the bounds on the eigenfunctions \eqref{0} and \eqref{k}, we deduce that for any $k\geq 1$,
    \begin{equation}\label{c}
        |c_{\sigma,k}^\eps|\le \begin{cases}\left(\frac{4L}{k\pi\sqrt\eps}+\frac{2\sqrt{2L}}{\tau(\lambda_k^\eps-\lambda_0^\eps)}\right)\left\|u_0\right\|_{L^2(-L,L)},& \sigma\ge0\\
        \left(\frac{4L}{k\pi\sqrt\eps}\exp\left(\frac{2|\sigma|}{\eps}\right)+\frac{2\sqrt{2L}}{\tau(\lambda_k^\eps-\lambda_0^\eps)}\right)\left\|u_0\right\|_{L^2(-L,L)},& \sigma<0
        \end{cases}.
    \end{equation}
    Finally, the upper bound \eqref{bior} gives 
       \begin{equation}
        \|\wt{h^\eps}\|_{L^2(0,\wh T)}\le \sum_{k\ge1} \frac{c}{\lambda_{\sigma,k}^\eps-\frac1{4\eps}}|c_{\sigma,k}|\exp\left(-\frac{\wh T}{8\eps}+\frac{6\kappa L^2}{\eps\wh T}\right)\left\|u_0\right\|_{L^2(-L,L)}.
    \end{equation}
    By the distribution of the eigenvalues \eqref{dist} and the bound on the coefficients \eqref{c}, this sum is finite, and there exists $c=c(\eps,\kappa,L,\sigma,T,m,\tau)>0$ a rational function of its arguments such that, \begin{equation}
        \|\wt{h^\eps}\|_{L^2(0,\wh T)}\le \begin{cases} c \exp\left(-\frac{\wh T}{8\eps}+\frac{6\kappa L^2}{\eps\wh T}\right)\left\|u_0\right\|_{L^2(-L,L)},&\sigma\ge0\\
        c \exp\left(-\frac{\wh T}{8\eps}+\frac{6\kappa L^2}{\eps\wh T}+\frac{2|\sigma|}{\eps}\right)\left\|u_0\right\|_{L^2(-L,L)},&\sigma<0
        \end{cases}.
    \end{equation}
    As $\|h^\eps\|_{L^2(\tau, T)}=\|\wt{h^\eps}\|_{L^2(0,\wh T)},$ the constructed control decays exponentially as $\eps$goes to $0$ whenever the term in the exponential is negative. If $\sigma\ge0$, this holds under the condition \[\wh T>4\sqrt3\sqrt\kappa L.\]
    For $\sigma<0$, the condition is instead \[\wh T>4(2|\sigma|+\sqrt{4\sigma^2+3\kappa L^2}).\]
    Since we have made the assumption \eqref{ratio}, it is always possible to find some $\kappa$ close enough to $1$ such that those conditions are met. 
  This concludes the proof of Lemma \ref{lem2}.
  Moreover, it implies that \begin{equation}
      \label{72}\|h^\eps\|_{L^2(\tau,T)}=O(e^{-\frac c\eps}).
  \end{equation}
  Combining \eqref{71} and \eqref{72} concludes the proof of the convergence rate \eqref{111}, and of Theorem \ref{thm2}.

   \section{Proof of the lower bound on $T_{\rm{unif}}$}
   \label{sec-lissy}
In this section we prove that  \[T_{\rm{unif}}\geq \begin{cases} (4\sqrt2-2)L,& \sigma\ge 0\\
(4\sqrt2-2)L+8|\sigma|,& \sigma<0\end{cases}.\]
    We follow the work from \cite{lissy2015explicit}, improving \cite[pp 106-109]{Coron07}, to give a lower bound on the norm of any control function steering a well-chosen initial datum to zero.\\
    We set $T>0, \eps>0$, and choose the initial datum of the control system \eqref{cont} as $\ue_0=\sech\left(\frac{x-\sigma}{2\eps}\right)\varphi^\eps_{\sigma,1}(x).$ Let $h^\eps(t)$ the associated optimal control function. By definition of the null-controllability cost, it verifies \begin{equation}\label{cost1}
        \|h^\eps\|_{L^2(0,T)}\le \cC(T,L,\sigma,\eps) \left\|\sech\left(\frac{x-\sigma}{2\eps}\right)\varphi_{\sigma,1}^\eps\right\|_{L^2(-L,L)}.
    \end{equation}
    Moreover, applying Lemma \ref{lemdual} on the whole time interval $[0,T]$, we have, for $k\ge0$ \begin{align}\label{dual1}
        \langle \ue(T),\psi_{\sigma,k}^\eps\rangle&\notag=e^{-\lambda_{\sigma,k}^\eps T}\langle \sech\left(\frac{x}{2\eps}\right)\varphi_{\sigma,1}^\eps,\psi_{\sigma,k}^\eps\rangle + \eps(\psi_{\sigma,k}^\eps)'(-L) \int_0^T h^\eps(t)e^{-\lambda_{\sigma,k}^\eps(T-t)}\,\dd t.\\
        &=e^{-\lambda_{\sigma,k}^\eps T}\langle \varphi_{\sigma,1}^\eps,\varphi_{\sigma,k}^\eps\rangle + \eps(\psi_{\sigma,k}^\eps)'(-L) \int_0^T h(t)e^{-\lambda_{\sigma,k}^\eps(T-t)}\,\dd t.
    \end{align}
    We introduce the function \begin{equation}
        v^\eps(z):=\int_{-\frac T2}^{\frac T2} h^\eps\left(t+\frac T2\right)e^{-izt}\,\dd t.
    \end{equation}
    $v^\eps$ is the Fourier transform of a compactly supported function, and is therefore entire.\\
    Since the operator $P^\eps$ acting on $H^1_0\cap H^2(-L,L)$ is self-adjoint, the family $(\varphi_{\sigma,k}^\eps)_{k\ge1}$ is an orthogonal family. 
    It then follows from \eqref{dual1} and the constraint $\ue(T)=0$ that\begin{equation}
        v^\eps(i\lambda_{\sigma,k}^\eps)=\begin{cases}
            -e^{\frac{-\lambda_{\sigma,1}^\eps T}{2}} \frac{\|\varphi_{\sigma,1}^\eps\|^2}{\eps(\psi_{\sigma,1}^\eps)'(-L)},& k=1\\
            0,& k\ne1.
        \end{cases}
    \end{equation}
    Using the bound \eqref{cost1}, we also have, for any $z\in\mathbb C$ \begin{align*}
        |v^\eps(z)|&\le \cC(T,L,\sigma,\eps)\sqrt T \exp\left(\frac{T|Im(z)|}{2}\right) \left\|\sech\left(\frac{x}{2\eps}\right)\varphi_{\sigma,1}^\eps\right\|_{L^2(-L,L)}
    \end{align*}
    We defined the rescaled entire function \[g^\eps(z):=v^\eps\left(\frac{z}{4\eps}\right),\]
    which immediately verifies, for $z\in \mathbb C$, \begin{equation}
        |g^\eps(z)|\le \cC(T,L,\sigma,\eps)\sqrt T\exp\left(\frac{T|Im(z)|}{8\eps}\right)\left\|\sech\left(\frac{x}{2\eps}\right)\varphi_{\sigma,1}^\eps\right\|_{L^2(-L,L)}.\label{type}
    \end{equation}
    Similarly to the proof of the upper bound, we rewrite the null-control problem in terms of moments of the control function $h^\eps$, and we express it as the inverse Fourier transform of some entire function with prescribed values at some key points. This point of view is well studied in complex analysis, notably with the Hadamard and Weyl factorization theorems. In this case, we will apply a representation theorem from \cite[p.56]{koosis1988logarithmic} to write, for $z$ in the upper half-plane, \begin{equation}\label{koosis}
        \ln|g^\eps(z)|=\sum_{\ell\ge1}\ln\left|\frac{z-a_\ell^\eps}{z-\overline{a_\ell^\eps}}\right|+\alpha Im(z) + \frac{Im(z)}{\pi}\int_{\R}\frac{\ln|g^\eps(s)|}{|s-z|^2}\,\dd s,
    \end{equation}
    where $\alpha$ is the type of the entire function $g^\eps$ and the $(a_\ell^\eps)_{\ell\ge 1}$ are the roots of $g^\eps$ in the upper half-plane. We apply this equality in $z=4i\eps\lambda_{\sigma,1}^\eps$, and want to find an upper bound on the right-hand side.\\
    For the second term, we know from the estimate \eqref{type} that the type of $g^\eps$ verifies $\alpha\le \frac{T}{8\eps}$, so that\begin{equation}\label{maj2}\alpha \Im(4i\eps\lambda_{\sigma,1}^\eps)\le \frac{\lambda_{\sigma,1}^\eps T}{2}.
    \end{equation}
    Applying the estimate \eqref{type} again with $s\in\R$, we get \begin{equation*}
        \ln|g^\eps(s)|\le\ln\left(\cC(T,L,\sigma,\eps)\sqrt T\left\|\sech\left(\frac{x}{2\eps}\right)\varphi_{\sigma,1}^\eps\right\|_{L^2(-L,L)}\right),
    \end{equation*}
    and we may directly deduce a bound for the third term:
    \begin{align}
        \frac{Im(4i\eps\lambda_{\sigma,1}^\eps)}{\pi}\int_\R \frac{\ln|g^\eps(s)|}{|s-4i\eps\lambda_{\sigma,1}^\eps|^2}\,\dd s&\le\notag \frac{4\eps\lambda_{\sigma,1}^\eps}{\pi}\ln\left(\cC(T,L,\sigma,\eps)\sqrt T\left\|\sech\left(\frac{x}{2\eps}\right)\varphi_{\sigma,1}^\eps\right\|_{L^2(-L,L)}\right)\int_\R\frac1{s^2+(4\eps\lambda_{\sigma,1}^\eps)^2}\,\dd s\\
        &\le \ln\left(\cC(T,L,\sigma,\eps)\sqrt T\left\|\sech\left(\frac{x}{2\eps}\right)\varphi_{\sigma,1}^\eps\right\|_{L^2(-L,L)}\right).\label{maj3}
    \end{align}
    Finally, for the first term, we only use the known roots of the form $4i\eps\lambda_{\sigma,k}^\eps,$ for $k\ge 2$, which provides the upper bound
    \begin{align}
         \sum_{\ell\ge 1} \ln\left|\frac{4i\eps\lambda_{\sigma,1}^\eps - a_\ell^\eps}{4i\eps\lambda_{\sigma,1}^\eps -\overline{a_\ell^\eps}}\right|\notag &\le \sum_{k\ge 2}\ln\left(\frac{\lambda_{\sigma,k}^\eps-\lambda_{\sigma,1}^\eps}{\lambda_{\sigma,k}^\eps+\lambda_{\sigma,1}^\eps}\right).
    \end{align}
    With the distribution of the eigenvalues \eqref{dist} and \eqref{gap}, this becomes \begin{align}
        \sum_{k\ge 2}\ln\left(\frac{\lambda_{\sigma,k}^\eps-\lambda_{\sigma,1}^\eps}{\lambda_{\sigma,k}^\eps+\lambda_{\sigma,1}^\eps}\right)&\le \sum_{k\ge 2} \ln\left(\frac{((k+1)^2-1)\frac{\eps\pi^2}{4L^2}}{\frac1{2\eps}+ (k^2+1)\frac{\eps\pi^2}{4L^2}}\right)\notag\\&\le \sum_{k\ge 2} \left(\frac{(k^2-1)\frac{\eps\pi^2}{4L^2}}{\frac1{2\eps}+(k^2+1)\frac{\eps\pi^2}{4L^2}}\right) - \ln\left(\frac{3\eps\pi^2}{4L^2}\right)\notag.
    \end{align}
    Comparing the series with an integral and integrating by parts then yields the upper bound \begin{equation}\label{maj1}
        \sum_{\ell\ge 1} \ln\left|\frac{4i\eps\lambda_{\sigma,1}^\eps - a_\ell^\eps}{4i\eps\lambda_{\sigma,1}^\eps -\overline{a_\ell^\eps}}\right|\le -\frac{\sqrt2L}\eps +2\ln\left(1+\frac{2L^2}{\pi^2\eps^2}\right)+\ln\left(\frac{4L^2}{3\eps\pi^2}\right)+2.
    \end{equation}
    Replacing $g^\eps(4i\eps\lambda_{\sigma,1}^\eps)$ with its expression and combining the bounds \eqref{maj2}, \eqref{maj3} and \eqref{maj1} then gives
    \begin{gather*}
        -\frac{\lambda_{\sigma,1}^\eps T}{2} \ln\left(\frac{\|\varphi_{\sigma,1}^\eps\|_{L^2(-L,L)}^2}{\eps(\psi_{\sigma,1}^\eps)'(-L)}\right)\le -\frac{\sqrt2L}\eps +2\ln\left(1+\frac{2L^2}{\pi^2\eps^2}\right)
        \\ \quad +\ln\left(\frac{4L^2}{3\eps\pi^2}\right)+2 + \frac{\lambda_{\sigma,1} T}{2} + \ln\left(\cC(T,L,\sigma,\eps)\sqrt T\left\|\sech\left(\frac{x-\sigma}{2\eps}\right)\varphi_{\sigma,1}^\eps\right\|_{L^2(-L,L)}\right)
    \end{gather*}
    Rearranging the terms gives the lower bound on the null-controllability cost \begin{equation}
        \cC(T,L,\sigma,\eps)\ge e^2\cdot\frac{4L^2}{3\eps\pi^2}\left(1+\frac{2L^2}{\pi^2\eps^2}\right)^2 \frac{\|\varphi_{\sigma,1}^\eps\|_{L^2(-L,L)}^2}{\eps(\psi_{\sigma,1}^\eps)'(-L) \left\|\sech\left(\frac{x-\sigma}{2\eps}\right)\varphi_{\sigma,1}^\eps\right\|_{L^2(-L,L)}}\exp\left(-\lambda_{\sigma,1}^\eps T+\frac{\sqrt2L}{\eps}\right).\label{fin-1}
    \end{equation}
    Finally, we use the expression of the eigenfunction $\varphi_{\sigma,1}^\eps$ given by \eqref{defphik}, which implies the existence of a constant $c>0$ such that \begin{equation}\frac{\|\varphi_{\sigma,1}^\eps\|_{L^2(-L,L)}^2}{\eps(\psi_{\sigma,1}^\eps)'(-L) \left\|\sech\left(\frac{x-\sigma}{2\eps}\right)\varphi_{\sigma,1}^\eps\right\|_{L^2(-L,L)}}\ge \begin{cases} c \exp\left(-\frac{L}{2\eps}\right), &\sigma\ge0\\c \exp\left(-\frac{L}{2\eps}-\frac{2|\sigma|}{\eps}\right),&\sigma<0\end{cases}.\end{equation}
    Plugging this estimate into the bound \eqref{fin-1} leads to \begin{equation}\left.\begin{aligned}
        \cC(T,L,\sigma,\eps)\ge ce^2\cdot\frac{4L^2}{3\eps\pi^2}\left(1+\frac{2L^2}{\pi^2\eps^2}\right)^2\exp\left(-\lambda_{\sigma,1}^\eps T+\frac{\sqrt2L}{\eps}-\frac L{2\eps}\right),\quad \text{if }\sigma\ge0,\\\cC(T,L,\sigma,\eps)\ge ce^2\cdot\frac{4L^2}{3\eps\pi^2}\left(1+\frac{2L^2}{\pi^2\eps^2}\right)^2\exp\left(-\lambda_{\sigma,1}^\eps T+\frac{\sqrt2L}{\eps}-\frac L{2\eps}-\frac{2|\sigma|}\eps\right),\quad \text{if }\sigma<0.\end{aligned}\right.
    \end{equation}Since $\lambda_{\sigma,1}^\eps\le \frac1{4\eps}+\eps\frac{\pi^2}{L^2}$, the null-controllability cost $\cC(T,L,\sigma,\eps)$ explodes as $\eps$ approaches $0$ if the quantity $-\frac T4 + \sqrt{2}L-\frac L2$ is positive if $\sigma\ge 0$ and if $-\frac T4 +\sqrt2L-\frac L2-\frac{2|\sigma|}\eps$ is positive if $\sigma<0$, namely whenever \begin{equation}
        T<\begin{cases} (4\sqrt2-2)L,& \sigma\ge 0\\
(4\sqrt2-2)L+8|\sigma|,& \sigma<0\end{cases}.
    \end{equation}
    The proof of the lower bound on $T_{\rm{unif}}$ is then concluded.

\section{The two-boundary control case}\label{sec ts}
\subsection{Setting and expected results}
Let $T>0$ be fixed. We consider the case $\sigma=0$ and wish to exploit the symmetry of the problem. By an abuse of notation, we drop the index $\sigma$ everywhere in the sequel.\\
We now consider weak solution of the same system, but in the case where one may act on the system from both endpoints on the interval. Namely, we study the control problem: \begin{equation}\label{cont2}\left\{\begin{aligned}
    \de_t \ue+\de_x(U^\eps \ue)&=\eps\de_x^2\ue,\quad &x\in (-L,L), t\in(0,T),\\
    \ue(t,-L)&=h_-^\eps(t),\quad &t\in(0,T),\\
    \ue(t,L)&=h_+^\eps(t),\quad &t\in (0,T),\\
    \ue(0,x)&=u_0(x),\quad &x\in(-L,L),
    \end{aligned}\right.
\end{equation}
The problem at stake remains the same. Given an initial condition $u_0\in L^2(-L,L)$, we want to build controls $h_-^\eps, h_+^\eps\in L^2(0,T)$ such that the solution of \eqref{cont2} verifies $u^\eps(T)\equiv0$. If we can always do this, we define the two-sided null-controllability cost as \[\cC_{TS}(T,L,\eps):=\sup_{u_0\in L^2, \|u_0\|_{L^2}=1} \inf \left\{\|(h_-^\eps,h_+^\eps)\|_{L^2(0,T)}: \ue(T)=0\right\}.\]
Again, the works of Fattorini and Russell, \cite{fattorini1974uniform},  show that the system \eqref{cont2} is null-controllable for any $T>0, L>0,\eps>0$, and we wish to find conditions on the control time $T$ so that the system is \textit{uniformly null-controllable}. As a reminder, it means we want to find constraints on $T,L$ such that \[\limsup_{\eps\to 0} \cC_{TS}(T,L,\eps)<+\infty.\]
Formally taking the limit as $\eps\to0$, we arrive at the control problem:
\begin{equation}\label{lim2}
    \left\{\begin{aligned}
    \de_t u-\de_x(\sgn(x)u)&=0,\\
    u(t,-L)&=h_-(t),\\
    u(t,L)&=h_+(t),\\
    u(0,x)&=u_0(x).
    \end{aligned}\right.
\end{equation}
The same analysis as in Section \ref{sec lim sys} gives that the system \eqref{lim2} is null-controllable if and only if $T>L$, and its control cost is given by \[\cC_{TS}(T,L,0)=\frac{\cC(T,L,0)}{\sqrt2}=\frac{\sqrt L}{\sqrt{T-L}},\] where $\cC(T,L,0)$ is the controllability cost of the limit system with only one boundary control, computed in Section \ref{sec lim sys}.\\

The null-control strategy is again to cancel out the mean of the solution in an arbitrarily small time, and then to set both controls to $0$ for a time interval of length $L$.\\

From the study of the limit system, it is a priori not obvious that controlling the system from both sides is easier than using only the left endpoint. However, morally speaking, when controlling from $x=-L$ , the transport term $U^\eps\de_x$ only helps to propagate the information from the left endpoint to the left half $[-L, 0]$. On the other half of the interval, the transport term goes towards the left, and we have to pay a price and use the viscous term to send the desired information to the right half $[0,L]$.\\

On the contrary, in the case of the two-sided control, the transport term is beneficial on the whole interval, as it propagates the data from the left endpoint to the left half of the interval, and the data from the right endpoint to the right half of the interval. We may therefore reasonably expect a better uniform null-controllability time in this setting.

\subsection{Statement of the result}

\begin{thm}\label{thm1-ts}
    There exists a minimal time $T_{\rm{unif}}^{\rm{TS}} < +\infty$ such that the system \eqref{cont2} is uniformly null-controllable for any time $T>T_{\rm{unif}}^{\rm{TS}}$. Moreover, this minimal time verifies \begin{equation}
        \label{tunif ts} T_{\rm{unif}}^{\rm{TS}}\in\left[1, 2\sqrt3\right]L.
    \end{equation}
\end{thm}

\begin{remark}
    There are two important features of this result to note, as a comparison to the control from one endpoint. First, the upper bound on the uniform controllability time is halved from $4\sqrt3L$ to $2\sqrt3L$, which is not obvious a priori. It now matches the upper bound obtained in \cite{lissy2012link} for an interval of width $L$. On top of that, the nontrivial lower bound on $T_{\rm{unif}}$ we had in Theorem \ref{thm1} is no longer present in this setting.
\end{remark}
\begin{remark}
    The case $\delta\ne0$ can also be treated, however it does not present the same symmetry properties which allow us to significantly improve the uniform controllability time. Without those properties, our approach would consist in only using the control on the "good" enpoint at which the derivatives of the eigenfunctions are larger, namely at $x=L$ if $\delta<0$ and $x=-L$ if $\delta>0$.\\
    With this method, we can improve Theorem \ref{thm1} to \[T_{\rm{unif}}^{\rm{TS}}<4\sqrt3L,\]
    for all $\delta\ne0$. It is currently unclear whether it is possible to use both controls to improve this upper bound.
\end{remark}

The proof of this result, done below in Section \ref{sec proof burg ts}, is performed through a constructive approach, and we therefore have a secondary result providing information on some admissible control functions $h_-^\eps, h_+^\eps$.

\begin{thm}\label{thm2 ts}
    Let $L>0, T>T^*=2\sqrt3L$. Then, there exists $C>0$ such that, for any $u_0\in L^2(-L,L),$ $\eps>0$, there exist control functions $h_-^{u_0,\eps}, h_+^{u_0,\eps}\in L^2(0,T)$ such that the solution $\ue$ of \eqref{cont} satisfies $\ue(T)=0$ and \begin{equation} \label{major ts}
        \|(h_-^{u_0,\eps}, h_+^{u_0,\eps})\|_{L^2(0,T)}\le \left(\frac{2\sqrt{2L}}{\sqrt{T-T^*}}+Ce^{-\frac C\eps}\right)\|u_0\|_{L^2(-L,L)}.
    \end{equation}
    As $\eps$ goes to $0$, these control functions converge in $L^2(0,T)$ towards the limit control \begin{equation}\label{7 ts}
        h_-^{u_0,0}(t)=h_+^{u_0,0}(t)=\begin{cases}
            -\frac1{T-T^*}\int_{-L}^Lu_0(x)\,\dd x,& t\le\frac{T-T^*}{2},\\
            0& t>\frac{T-T^*}2.
        \end{cases}
    \end{equation}
\end{thm}
\subsection{Scheme of the proof}\label{sec proof burg ts}
As hinted by the structure of the limit control, we again proceed in two steps. First, in an arbitrarily small time $\tau>0$, we build a simple control that kills the first eigenmode. More precisely, we have the following lemma.

\begin{Lemma}\label{lem1-ts}
Let $T>0$,  $\tau \in (0,T)$. 
  For any $\eps>0$ and any initial datum $u_0\in L^2(-L,L)$, there is a control  $(h_-^{1,\eps}, h_+^{1,\eps})\in L^2(0,\tau)$ verifying: 
    \begin{align}
        \|(h_-^{1,\eps}, h_+^{1,\eps})\|_{L^2(0,\tau)}\le\frac{2\sqrt{L}}{\sqrt\tau} \|u_0\|_{L^2(-L,L)}, \label{cout1 ts}
    \end{align}
   such that the solution $\ue$ of \eqref{cont2} satisfies
   at time $\tau$, 
   for any $k \in \N$,
    \begin{equation}\label{ytau-ts}
        \langle \ue(\tau),\psi_k^\eps\rangle=\begin{cases}e^{-\lambda_k^\eps\tau}\langle u_0,\psi_k^\eps\rangle - \frac{(\psi_k^\eps)'(-L)}{(\psi_0^\eps)'(-L)}\cdot\frac{e^{-\lambda_0^\eps\tau}\langle u_0,\psi_0^\eps\rangle}{\tau} \int_0^\tau e^{-(\lambda_k^\eps-\lambda_0^\eps)t}\,\dd t,&\text{ for }k\text{ even,}\\e^{-\lambda_k^\eps\tau}\langle u_0,\psi_k^\eps\rangle,&\text{ for }k\text{ odd.}
        \end{cases}
    \end{equation}\\
    In particular, for $k=0$, 
    \begin{equation}   \langle \ue(\tau),\psi_0^\eps\rangle=0.\label{mode1-ts}
        \end{equation}
\end{Lemma}

Then, starting from that state $\ue(\tau)$, we may steer the solution to $0$ with a small cost provided the remaining time $T-\tau$ is long enough:

\begin{Lemma}\label{lem2-ts}
Let $T>0$,  $\tau \in (0,T)$ with 
\begin{equation} \label{ratio-ts}
  T-\tau>T^*=2\sqrt{3}L.  
\end{equation}
Then there exists $C>0$ such that for any $u_0\in L^2(-L,L)$, for any $\eps>0$ , 
there exists a control $(h_-^{2,\eps}, h_+^{2,\eps})\in L^2(\tau,T)$ with 
    \begin{equation}\label{cout2-ts}\|(h_-^{2,\eps}, h_+^{2,\eps})\|_{L^2(\tau,T)}\le Ce^{-\frac C\eps}\left\|u_0\right\|_{L^2(-L,L)},\end{equation}
    such that 
    the solution $\ue$ of \eqref{cont2} on the time interval $(\tau,T)$ starting at time $\tau $ with the initial data $\ue(\tau)$ given by the formula  \eqref{ytau-ts},  for  $k \in \N$,
    satisfies 
    $\ue(T)=0$.
\end{Lemma}
We assume for the moment that these two Lemmas are true, and we shall prove them in the sections below.\\
Let $u_0\in L^2(-L,L)$, $\eps>0$, and $T>T^*=2\sqrt3L$. We choose $\tau=\frac{T-T^*}{2}$, so that \[T-\tau=\frac{T+T^*}2>T^*.\]
We may now apply Lemmas \ref{lem1-ts} and \ref{lem2-ts} and consider the two resulting controls $(h^{1,\eps}_-,h^{1,\eps}_+)\in L^2(0,\tau)$, $(h^{2,\eps}_-,h^{2,\eps}_+)\in L^2(\tau,T)$.\\
We construct the control as \begin{equation}
    (h^\eps_-,h^\eps_+)(t) =\begin{cases}
        (h_-^{1,\eps}, h_+^{1,\eps})(t)&\text{if } 0 < t<\tau,\\
    (h_-^{2,\eps}, h_+^{2,\eps})(t)&\text{if }\tau <  t < T.
    \end{cases}
\end{equation}
This choice of control steers the solution of \eqref{cont2} from $u_0$ to $\ue(T)=0$.\\
We may also obtain direct estimates on the norm of the control as follows:
\begin{align}
    \notag\|(h_-^\eps,h_+^\eps)\|_{L^2(0,T)}&\le \|(h_-^{1,\eps},h_+^{1,\eps})\|_{L^2(0,\tau)}+\|(h_-^{2,\eps}, h_+^{2,\eps})\|_{L^2(\tau,T)}\\
    &\le \left(\frac{2\sqrt{2L}}{\sqrt{T-T^*}}+Ce^{-\frac C\eps}\right)\left\|u_0\right\|_{L^2(-L,L)}\label{48 ts}.
\end{align}
This proves the estimate \eqref{major ts}. Moreover, it proves that the system \eqref{cont2} is therefore uniformly null-controllable for $T>T^*$, which proves the upper bound $T_{\rm{unif}}\le 2\sqrt3L$.

\subsection{Action of the control on the eigenmodes}
To describe the impact of the control on each eigenmode of the solution, we begin with a preliminary lemma exploiting duality equalities.
\begin{Lemma}\label{lemdual2}
    Let $\eps>0$, $t_1, t_2\in[0,T]$, and $k\ge0$. The evolution of the $k$-th mode of the solution $\ue$ of the control problem \eqref{cont2} is prescribed by:\begin{itemize}
        \item If $k$ is even, \begin{equation}\label{dual even}
            \langle \ue(t_2),\psi_k^\eps\rangle=e^{-\lambda_k^\eps(t_2-t_1)}\langle \ue(t_1),\psi_k^\eps\rangle+\eps(\psi_k^\eps)'(-L)\int_{t_1}^{t_2}e^{-\lambda_k^\eps (t-t_1)}(h_-+h_+)(t_1+t_2-t)\,\dd t.
        \end{equation}
        \item If $k$ is odd, \begin{equation}\label{dual odd}
            \langle \ue(t_2),\psi_k^\eps\rangle=e^{-\lambda_k^\eps(t_2-t_1)}\langle \ue(t_1),\psi_k^\eps\rangle+\eps(\psi_k^\eps)'(-L)\int_{t_1}^{t_2}e^{-\lambda_k^\eps (t-t_1)}(h_--h_+)(t_1+t_2-t)\,\dd t.
        \end{equation}
    \end{itemize}
\end{Lemma}
\begin{proof}
    Let $k\in\N$. We proceed as in the proof of Lemma \ref{lemdual} in the case of one boundary control. Two boundary terms now arise in the duality relation, which reads: \begin{align}
        \langle \ue(t_2),\psi_k^\eps\rangle=&e^{-\lambda_k^\eps(t_2-t_1)}\langle \ue(t_1),\psi_k^\eps\rangle +\notag \\ &\eps\int_{t_1}{t_2} e^{-\lambda_k^\eps(t-t_1)}\left(h_-(t_1+t_2-t) (\psi_k^\eps)'(-L) - h_+(t_1+t_2-t) (\psi_k^\eps)'(L)\right)\,\dd t.
    \end{align}
    To conclude, we simply notice that the eigenfunction $\psi_k^\eps$ is even if $k$ is even, and odd if $k$ is odd, as seen in the definitions of the eigenfunctions \eqref{defphi0} and \eqref{defphik}. We therefore have $(\psi_k^\eps)'(L)=(-1)^{k+1} (\psi_k^\eps)'(-L)$ and the claim is proven.
\end{proof}

Note that the action of the control on the odd and even eigenmodes are completely independent! The even component of the control only acts on the even eigenmodes, and the odd component on the odd eigenmodes. This will allow us to treat this null-controllability problem as two easier independent moment problems.
Since each moment problem has more relaxed constraints than the one in the previous sections, we expect a significantly lesser controllability cost.\begin{remark}
    The symmetry of the terms involved in the operator yields this even or odd relation between the derivatives of the eigenfunctions at both endpoints. In the shifted case $\sigma\ne0$, this does no longer hold. There is no simple relation between $(\psi_{\sigma,k}^\eps)'(-L)$ and $(\psi_{\sigma,k}^\eps)'(L)$, and it is unclear whether it is possible to combine both endpoints in order to obtain simpler decoupled moment problems.
\end{remark}

\subsection{Step 1: Killing the metastability, proof of Lemma \ref{lem1-ts}}

Let $u_0\in L^2(-L,L), \tau>0$. By Lemma \ref{lemdual2}, the action of the control on the first eigenmode is given by \[\langle\ue(\tau),\psi_0^\eps\rangle = e^{-\lambda_0^\eps\tau} \langle u_0,\psi_0^\eps\rangle +\eps(\psi_0^\eps)'(-L)\int_0^\tau (h^\eps_-(t)+h^\eps_+(t)) e^{-\lambda_0^\eps(\tau-t)}\,\dd t.\]
We may simply choose \begin{equation}\label{meta-ts}
    h_-^\eps(t)=h_+^\eps(t)=-\frac{e^{-\lambda_0^\eps t} \langle u_0,\psi_0^\eps\rangle}{2\tau \eps(\psi_0^\eps)'(-L)},
\end{equation}
which directly yields $\langle u^\eps(\tau),\psi_0^\eps\rangle=0.$\\Moreover, it holds \[\|(h_-^\eps,h_+^\eps)\|_{L^2(0,\tau)} \le \frac{\sqrt2 \|u_0\|_{L^2(-L,L)} \|\psi_0^\eps\|_{L^2(-L,L)}}{2\tau \eps |(\psi_0^\eps)'(-L)|}.\]
Using the estimates on the first eigenfunction \eqref{0}, it follows
\begin{equation}
    \|(h_-^\eps,h_+^\eps)\|_{L^2(0,\tau)}\le \frac{2\sqrt{L}}{\tau} \|u_0\|_{L^2(-L,L)}.
\end{equation}
We may also keep track of the impact of this control on the other eigenmodes by using again Lemma \ref{lemdual2}.\\
For $k\in \N$ odd, since we chose the control such that $h_-^\eps(t)-h_+^\eps(t)=0$, relation \eqref{dual odd} immediately gives \begin{equation}
    \langle \ue(\tau),\psi_k^\eps\rangle=e^{-\lambda_k^\eps\tau}\langle u_0,\psi_k^\eps\rangle.
\end{equation}
For $k\in \N$ even, we instead directly plug the expression of the control \eqref{meta-ts} into the duality equality \eqref{dual even}, and it leads to the desired relation \eqref{ytau-ts}. Lemma \ref{lem1-ts} is now fully proven.

\subsection{Step 2: Controlling the rest to zero, proof of Lemma \ref{lem2-ts}}
Let $T>0$, and $\tau\in (0,T)$ such that \begin{equation}\label{ratio2-ts}T-\tau>T^*=2\sqrt3L.\end{equation}
Let $u_0\in L^2(-L,L)$, and $\ue(\tau)$ the intermediary state built in Lemma \ref{lem1-ts} and defined by \eqref{ytau-ts}.\\
For convenience, we introduce $\wh T:=T-\tau, \wt{\ue}(t)=\ue(t+\tau), \wt h_\pm^\eps(t):=\wt h_\pm^\eps(t+\tau).$
The control problem \eqref{cont2} can then be rewritten as
\begin{equation}\label{contrest}\left\{\begin{aligned}
    \de_t\wt\ue+\de_x(U^\eps\wt\ue)&=\eps\de_x^2\wt\ue,\\
    \wt\ue(t,-L)&=\wt h_-^\eps(t),\\
    \wt\ue(t,L)&=\wt h_+^\eps(t),\\
    \wt\ue(0,x)&=\ue(\tau),
\end{aligned}\right.\end{equation}and we wish to build controls $\wt h_\pm^\eps\in L^2(0,\wh T)$ such that $\wt\ue(\wh T)=0.$\\
Since the $(\psi_k^\eps)_{k\in\N}$ form a basis of $L^2(-L,L)$, this is equivalent to reaching $\langle \wt\ue(\wh T),\psi_k^\eps\rangle=0$ for all $k\in\N$. Therefore, Lemma \ref{lemdual2} allows us to rewrite the null-controllability issue as a moment problem. Namely, we want to find controls $\wt h_\pm\in L^2(0,\wh T)$ such that:\begin{itemize}
    \item For all $k\in\N$ even, \begin{equation}\label{momeven}
        \int_0^{\wh T} (\wt h_-^\eps(t)+\wt h_+^\eps(t))e^{\lambda_k^\eps t}\,\dd t=-\frac{\langle \ue(\tau), \psi_k^\eps\rangle}{\eps(\psi_k^\eps)'(-L))}.
    \end{equation}
    \item For all $k\in\N$ odd, \begin{equation}\label{momodd}
        \int_0^{\wh T} (\wt h_-^\eps(t)-\wt h_+^\eps(t))e^{\lambda_k^\eps t}\,\dd t=-\frac{\langle \ue(\tau), \psi_k^\eps\rangle}{\eps(\psi_k^\eps)'(-L))}.
    \end{equation}
\end{itemize}
As opposed to Section \ref{sec-upper} where we had one moment problem on the control function $\wt h^\eps$, we now have two completely independent ones, for $\wt h_-^\eps+\wt h_+^\eps$ and for $\wt h_-^\eps-\wt h_+^\eps$. Building a full bi-orthogonal family to the family of exponentials $(e^{\lambda_k^\eps t})_{k\in\N}$ is no longer necessary, and we can instead use two "half" bi-orthogonal families. This is the object of the following lemma.

\begin{Lemma}
    Let $\wh T>0, \eps>0$. There exists a family $(q_k^\eps)_{k\ge1}\subset L^2(0,\wh T)$ such that:
    \begin{align}
        \text{For all }j\ge 0,k\ge 1,&\quad \int_0^{\wh T} e^{\lambda_{2j}^\eps t}q_{2k}^\eps(t)\,\dd t=\delta_{j,k},\\
        \text{For all }j\ge 0, k\ge 0, &\quad \int_0^{\wh T} e^{2\lambda_{2j+1}^\eps t} q_{2k+1}^\eps(t)\,\dd t=\delta_{j,k}.
    \end{align}
    Moreover, for every $\kappa>1$, there exists such a family and a constant $c=c(\eps,\kappa,L,\wh T)$ which is a rational function of its arguments such that, for all $k\ge 1$, \begin{equation} \label{bior-ts}
        \|q_k^\eps\|_{L^2(0,\wh T)}\le \frac{c}{\lambda_k^\eps-\frac1{4\eps}} \exp\left(-\frac{\wh T}{8\eps}+\frac{3\kappa L^2}{2\eps \wh T}\right).
    \end{equation}
\end{Lemma}
\begin{proof}
    We apply the Theorem \ref{thm bior}, proven in Section \ref{app bior}, twice.\\
    For the "even" bi-orthogonal family, we note that the family $(\lambda_{2k}^\eps)_{k\ge 0}$ verifies the assumptions from Section \ref{app bior}, with \[C_1=\frac1{4\eps}, C_2=\frac{\pi^2\eps}{L^2}.\]
    For the "odd" bi-orthogonal family, we do not need to worry about not re-exciting the first eigenmode, and can apply Theorem \ref{thm bior} with any choice of $\lambda_0$, with the same choice of constants \[C_1=\frac1{4\eps}, C_2=\frac{\pi^2\eps}{L^2}.\]
    This directly provides us the bound: \begin{equation}
        \forall k\ge 1, \|q_k^\eps\|_{L^2(0,\wh T)}\le \frac{c}{\lambda_k^{\eps}-\frac1{4\eps}} \exp\left(-\frac{\wh T}{8\eps} + \frac{3\kappa L^2}{2\eps \wh T}\right).
    \end{equation}
    which results in the estimate \eqref{bior-ts}.
\end{proof}
Let $\kappa>1$ to be determined later.
We then construct our controls $\wt h_\pm$ in the following way: \begin{align}
    \wt h_-^\eps(t)+\wt h_+^\eps(t)=&-\sum_{k\ge 1} c^\eps_{2k}q_{2k}^\eps(t),\\
    \wt h_-^\eps(t)-\wt h_+^\eps(t)=&-\sum_{k\ge0} c^\eps_{2k+1} q_{2k+1}^\eps(t),
\end{align}
where the coefficients $c_k^\eps$ are defined  by \[c_k^\eps:=\frac{\langle \ue(\tau),\psi_k^\eps\rangle}{\eps(\psi_k^\eps)'(-L)}.\]
That is, we choose the controls \begin{equation}
    \wt h_-^\eps(t)=-\frac12\sum_{k\ge 1}c_k^\eps q_k^\eps(t),\quad \wt h_+^\eps(t)=-\frac12\sum_{k\ge 1} (-1)^kc_k^\eps q_k^\eps(t).
\end{equation}
The moment constraints \eqref{momeven} and \eqref{momodd} are then trivially verified, so these controls, provided they are well-defined and are in $L^2(0,\wh T)$, steer the solution of \eqref{contrest} from $\ue(\tau)$ to $\wt\ue(\wh T)=0$.\\
To obtain suitable $L^2$ estimates on the control, we begin by providing upper bounds on the coefficients $c_k^\eps$.\begin{itemize}
    \item For $k\ge 1$ odd, the relation \eqref{ytau-ts} allows us to rewrite the coefficient as \[c_k^\eps=\frac{\langle\ue(\tau),\psi_k^\eps\rangle}{\eps(\psi_k^\eps)'(-L)}=e^{-\lambda_k^\eps\tau} \frac{\langle u_0,\psi_k^\eps\rangle}{\eps (\psi_k^\eps)'(-L)}.\]
    Applying Cauchy-Schwarz' inequality then directly yields \[
        |c_k^\eps|\le \frac{\|\psi_k^\eps\|_{L^2(-L,L)}}{|\eps(\psi_k^\eps)'(-L)|}\|u_0\|_{L^2(-L,L)}.
    \]
    Using the estimate on the eigenfunctions \eqref{k}, we conclude \begin{equation}
        |c_k^\eps|\le \frac{4L}{k\pi\sqrt\eps}\|u_0\|_{L^2(-L,L)} \label{coef odd}.
    \end{equation}
    \item For $k\ge 1$ even, we again use the relation \eqref{ytau-ts} and the coefficient reads \[c_k^\eps=\frac{\langle\ue(\tau),\psi_k^\eps\rangle}{\eps(\psi_k^\eps)'(-L)}=e^{-\lambda_k^\eps\tau} \frac{\langle u_0,\psi_k^\eps\rangle}{\eps (\psi_k^\eps)'(-L)}+e^{-\lambda_0^\eps\tau} \frac{\langle u_0,\psi_0^\eps\rangle}{\eps(\psi_0^\eps)'(-L)}\cdot\frac{\int_0^\tau e^{-(\lambda_k^\eps-\lambda_0^\eps)t}\,\dd t}{\tau}.\]
    Applying Cauchy-Schwarz' inequality and using rough upper bounds gives \[|c_k^\eps|\le \left(\frac{\|\psi_k^\eps\|_{L^2(-L,L)}}{|\eps(\psi_k^\eps)'(-L)|}+\frac{\|\psi_0^\eps\|_{L^2(-L,L)}}{|\eps(\psi_0^\eps)'(-L)|}\right) \|u_0\|_{L^2(-L,L)}.\]
    Again, the estimates on the eigenfunctions \eqref{k} as well as \eqref{0} allow us to conclude:
    \begin{equation}
        |c_k^\eps|\le \left(\frac{4L}{k\pi\sqrt\eps} + 2\sqrt{2L}\right)\|u_0\|_{L^2(-L,L)}. \label{coef even}
    \end{equation}
\end{itemize}
Finally, we may estimate the $L^2$ norm of the control. Namely, we have: \begin{equation}
    \|(\wt h_-^\eps, \wt h_+^\eps)\|_{L^2(0,\wh T)}\le\frac{\sqrt2}2 \sum_{k\ge 1} |c_k^\eps| \|q_k^\eps\|_{L^2(0,\wh T)}.
\end{equation}
Using the bound on the biorthogonal family \eqref{bior-ts}, it follows \begin{equation}
    \|(\wt h_-^\eps, \wt h_+^\eps)\|_{L^2(0,\wh T)}\le c\exp\left(-\frac{\wh T}{8\eps}+\frac{3\kappa L^2}{2\eps \wh T}\right) \|u_0\|_{L^2(-L,L)} \sum_{k\ge 1} \frac{|c_k^\eps|}{(\lambda_k^\eps-\frac1{4\eps})\|u_0\|_{L^2(-L,L)}}.
\end{equation}
Combining estimates \eqref{coef odd}, \eqref{coef even} and the distribution of the eigenvalues \eqref{dist}, this sum converges and can be bounded by a constant $c=c(\tau,\eps,L,\wh T)$ which is a rational fraction of its arguments, independent of the initial condition $u_0$.\\
Finally, it follows \begin{equation}
    \|(\wt h_-^\eps, \wt h_+^\eps)\|_{L^2(0,\wh T)}\le c\exp\left(-\frac{\wh T}{8\eps}+\frac{3\kappa L^2}{2\eps \wh T}\right) \|u_0\|_{L^2(-L,L)}.
\end{equation}
By the assumption \eqref{ratio2-ts}, $\wh T>2\sqrt3L$, so we can choose $\kappa>1$ such that $\wh T>2\sqrt{3\kappa}L$, and therefore \[-\frac{\wh T}{8}+\frac{3\kappa L^2}{2\wh T}<0.\]
This concludes the proof of Lemma \ref{lem2-ts}

\subsection{Absence of short-time obstruction}\label{sec-lissy2}
It is natural to wonder whether a non-trivial lower bound as the one obtained in Section \ref{sec-lissy} can be reached through a similar method in the context where the control may act on both endpoints.\\

We set $T>0, \eps>0$, and consider again the initial datum of the control system \eqref{cont2} given by $u_0^\eps=\sech\left(\frac{x}{2\eps}\right)\varphi_1^\eps$. Let $(h_-^\eps(t), h_+^\eps(t))$ the corresponding optimal control functions. It holds \begin{equation}\label{cost1 ts}
    \|(h_-^\eps, h_+^\eps)\|_{L^2(0,T)} \le \cC_{TS}(T,L,\eps) \left\|\sech\left(\frac{x}{2\eps}\right)\varphi_1^\eps\right\|_{L^2(-L,L)}.
\end{equation}
Applying \eqref{lemdual2}, the moment formulation of the null-controllability issue is as follows:
\begin{itemize}
    \item For $k\in\N$ even, \begin{equation}\label{obst even}
        \int_0^T (h_-^\eps+h_+^\eps)(t)e^{\lambda_k^\eps t}\,\dd t= -\frac{\left\langle\sech\left(\frac x{2\eps}\right)\varphi_1^\eps,\psi_k^\eps\right\rangle}{\eps(\psi_k^\eps)'(-L)}.
    \end{equation}
    \item For $k\in \N$ odd, \begin{equation} \label{obst odd}
        \int_0^T (h_-^\eps-h_+^\eps)(t)e^{\lambda_k^\eps t}\,\dd t= -\frac{\left\langle\sech\left(\frac x{2\eps}\right)\varphi_1^\eps,\psi_k^\eps\right\rangle}{\eps(\psi_k^\eps)'(-L)}.
    \end{equation}
\end{itemize}
From the definition of the eigenfunctions $(\varphi_k^\eps)$ and $(\psi_k^\eps)$, we have \[\left\langle\sech\left(\frac x{2\eps}\right)\varphi_1^\eps,\psi_k^\eps\right\rangle= \langle\varphi_1^\eps, \varphi_k^\eps\rangle.\]
Therefore, since the $(\varphi_k^\eps)$ form a Hilbert basis of $L^2(-L,L)$, it follows from \eqref{obst even} and \eqref{obst odd} that \begin{align}
    &\text{For $k$ even,}\quad \int_0^T (h_-^\eps+h_+^\eps)\label{obst even 2}(t)e^{\lambda_k^\eps t}\,\dd t=0.\\
    &\text{For $k$ odd,}\quad \int_0^T (h_-^\eps-h_+^\eps)\label{obst odd 2}(t)e^{\lambda_k^\eps t}\,\dd t= -\frac{\|\varphi_1^\eps\|^2_{L^2(-L,L)}}{\eps(\psi_1^\eps)'(-L)} \delta_{1,k}.
\end{align}
Note that the situation is already much different from the case of Section \ref{sec-lissy}, as there are no exploitable constraints on the even part of the control $h_-^\eps+h_+^\eps$, and the conditions on $h_-^\eps-h_+^\eps$ are similar, but only involve the odd eigenvalues.\\

In what follows, we do not make use of \eqref{obst even 2} to get a lower bound on the cost of the control, as those constraints can easily be met by choosing a skew-symmetric control such that $h_-^\eps=-h_+^\eps$. We therefore only focus on the symmetric part of the control.\\

We introduce the entire function \begin{equation}
    v(z):=\int_{-\frac T2}^{\frac T2} (h_-^\eps-h_+^\eps)\left(t+\frac T2\right) e^{-izt}\,\dd t.
\end{equation}
The conditions \eqref{obst odd 2} read, for $k\in\N$ odd,  \begin{equation}
    v(i\lambda_k^\eps)=\begin{cases}
        -e^{\frac{-\lambda_1^\eps T}{2}} \frac{\|\varphi_1^\eps\|^2}{\eps(\psi_1^\eps)'(-L)},& k=1\\
            0,& k\ne1.
    \end{cases}
\end{equation}
From the bound \eqref{cost1 ts}, we obtain:
\begin{equation}
    \|h_-^\eps-h_+^\eps\|_{L^2(0,T)}\le \sqrt2 \|(h_-^\eps, h_+^\eps)\|_{L^2(0,T)}\le \sqrt2 \cC_{TS}(T,L,\eps) \left\|\sech\left(\frac{x}{2\eps}\right)\varphi_1^\eps\right\|_{L^2(-L,L)},
\end{equation}
which allows us to directly derive the following estimate on $v(z),$ for $z\in\C$:
\begin{equation}
    |v(z)|\le \sqrt{2T} \cC_{TS}(T,L,\eps)\exp\left(\frac{T|Im(z)|}2\right) \left\|\sech\left(\frac{x}{2\eps}\right)\varphi_1^\eps\right\|_{L^2(-L,L)}.
\end{equation}
We again rescale this function by letting $f(z):=v\left(\frac z{4\eps}\right)$, and invoke the same representation theorem as in Section \ref{sec-lissy}:
\begin{equation}\label{koosis ts}
        \ln|f(z)|=\sum_{\ell\ge1}\ln\left|\frac{z-a_\ell}{z-\overline{a_\ell}}\right|+\sigma Im(z) + \frac{Im(z)}{\pi}\int_{\R}\frac{\ln|f(s)|}{|s-z|^2}\,\dd s,
    \end{equation}
where $\sigma$ is the type of the entire function $f$ and the $(a_\ell)_{\ell\ge 1}$ are the roots of $f$ in the upper half-plane.
Our goal is to estimate both sides of this inequality for $z=4i\eps\lambda_1^\eps$, and deduce a lower bound on the controllability cost from it.\\
For the left-hand side, the calculations are direct and we immediately have \begin{equation}\label{lhs ts}
    \ln|f(4i\eps\lambda_1^\eps)|=\ln|v(i\lambda_1^\eps)|=-\frac{\lambda_1^\eps T}{2} \ln\left(\frac{\|\varphi_1^\eps\|_{L^2(-L,L)}^2}{\eps(\psi_1^\eps)'(-L)}\right).
\end{equation}
For the right-hand side, the analysis is close to what was done with only one control. For the second and third term, we proceed exactly the same way and obtain \begin{align}
    \sigma \Im(4i\eps\lambda_1^\eps)&\le \frac{\lambda_1^\eps T}{2},\label{rhs 2 ts}\\
    \frac{Im(4i\eps\lambda_1^\eps)}{\pi}\int_\R \frac{\ln|f(s)|}{|s-4i\eps\lambda_1^\eps|^2}\,\dd s&\le \ln\left(\cC_{TS}(T,L,\eps)\sqrt{2T}\left\|\sech\left(\frac{x}{2\eps}\right)\varphi_1^\eps\right\|_{L^2(-L,L)}\right)\label{rhs 3 ts}.
\end{align}
For the first term of the right-hand side, the only zeros of $f$ in the upper half-plane that we may use are the $(4i\eps\lambda_{2k+1}^\eps)_{k\ge 1}.$ We are then led to study the sum \begin{equation}
    \sum_{\ell \ge 1}\ln\left|\frac{4i\eps\lambda_1^\eps - a_\ell}{4i\eps\lambda_1^\eps-\overline{a_\ell}}\right|\le \sum_{k\ge 1} \ln\left(\frac{\lambda_{2k+1}^\eps-\lambda_1^\eps}{\lambda_{2k+1}^\eps+\lambda_1^\eps}\right).
\end{equation}
Exploiting the distribution of the eigenvalues \eqref{dist} and \eqref{gap}, we have the upper bound \begin{equation}
    \sum_{k\ge 1} \ln\left(\frac{\lambda_{2k+1}^\eps-\lambda_1^\eps}{\lambda_{2k+1}^\eps+\lambda_1^\eps}\right) \le \sum_{k\ge 1} \ln\left(\frac{(4(k+1)^2-1)\frac{\eps\pi^2}{4L^2}}{\frac1{2\eps}+ ((2k+1)^2+1)\frac{\eps\pi^2}{4L^2}}\right).
\end{equation}
We may then use a series-integral comparison, and similarly to Section \ref{sec-lissy}, obtain an upper bound of the form \begin{equation}\label{rhs 1 ts}
    \sum_{\ell \ge 1}\ln\left|\frac{4i\eps\lambda_1^\eps - a_\ell}{4i\eps\lambda_1^\eps-\overline{a_\ell}}\right|\le -\frac{\sqrt2L}{2\eps} + \ln(c(L,\eps)),
\end{equation}
where $c(L,\eps)$ is a rational fraction of its arguments.\\
We may now combine all the estimates \eqref{lhs ts}, \eqref{rhs 2 ts}, \eqref{rhs 3 ts} and \eqref{rhs 1 ts} to obtain \begin{align}
    -\frac{\lambda_1^\eps T}{2} \ln\left(\frac{\|\varphi_1^\eps\|_{L^2(-L,L)}^2}{\eps(\psi_1^\eps)'(-L)}\right)\le -\frac{\sqrt2L}{2\eps} &+\notag \frac{\lambda_1^\eps T}2 +\ln(c(L,\eps)) +\\&+ \ln\left(\cC_{TS}(T,L,\eps)\sqrt{2T}\left\|\sech\left(\frac{x}{2\eps}\right)\varphi_1^\eps\right\|_{L^2(-L,L)}\right).
\end{align}
Rearranging the terms to isolate $\cC_{TS}(T,L,\eps)$ then gives \begin{equation}
    \cC_{TS}(T,L,\eps)\ge \frac{c(L,\eps)}{\sqrt{2T}}\cdot\frac{\|\varphi_1^\eps\|_{L^2(-L,L)}^2}{\eps(\psi_1^\eps)'(-L) \left\|\sech\left(\frac{x}{2\eps}\right)\varphi_1^\eps\right\|_{L^2(-L,L)}}\exp\left(-\lambda_1^\eps T+\frac{\sqrt2L}{2\eps}\right).\label{fin ts}
\end{equation}
To reach a conclusion, we use the definition of $\varphi_1^\eps$ in \eqref{defphik} as well as properties of the eigenfunction $\psi_1^\eps$, and it follows that there exists $c(L,\eps)$ a rational function of its arguments such that\begin{equation}
    \frac{\|\varphi_1^\eps\|_{L^2(-L,L)}^2}{\eps(\psi_1^\eps)'(-L) \left\|\sech\left(\frac{x}{2\eps}\right)\varphi_1^\eps\right\|_{L^2(-L,L)}}\ge c(L,\eps) \exp\left(-\frac{L}{2\eps}\right).
\end{equation}
Plugging this inequality into \eqref{fin ts} leads to \begin{equation}
    \cC_{TS}(T,L,\eps)\ge \frac{c(L,\eps)}{\sqrt T} \exp\left(-\lambda_1^\eps T + \frac{\sqrt2L}{2\eps} - \frac{L}{2\eps}\right).
\end{equation}
Since $\lambda_1^\eps\le\frac1{4\eps}+\frac{\pi^2\eps}{4L^2}$, the exponential explodes whenever \begin{equation}-\frac T4 + \frac{\sqrt{2}-1}{2}L>0,\end{equation} and therefore we may not have uniform controllability for $T<(2\sqrt2-2)L.$\\
However, $2\sqrt2-2\simeq0.83 <1$, so this analysis does not improve the necessary lower bound $T>L$ for uniform controllability that we obtained through the study of the limit system.\\

One may wonder whether a short-time obstruction can still be achieved by choosing a different initial condition that imposes constraints on the symmetric part of the control as well, for instance taking \[u^\eps_0=\sech\left(\frac{x}{2\eps}\right) (\varphi_1^\eps+\varphi_2^\eps).\]
Still, the fact that the moment constraints on the even modes and the odd modes are completely decoupled makes it so that we can not reach a better bound through this approach. Performing the same analysis would simply lead on a lower bound with the same exponential term on $\|h_-^\eps-h_+^\eps\|_{L^2(0,T)}$ and $\|h_-^\eps+h_+^\eps\|_{L^2(0,T)}$.

\appendix

\section{Construction of the bi-orthogonal families}\label{app bior}

In this appendix we provide a construction of a bi-orthogonal family, following the one proposed by Tenenbaum and Tucsnak in \cite{tenenbaum2007new}, and adapting it in order to incorporate the dissipation term, while making sure to not re-excite the first eigenmode.\newline\newline
\textbf{Spectral assumptions.}\\
Let $C_1,C_2,r>0$. Let $(\lambda_k)_{k\ge 0}$ a family of distinct real numbers such that:\begin{itemize}
    \item The sub-family $(\lambda_k)_{k\ge 1}$ follows the repartition:
    \begin{equation}\label{cond lambd1}
        \forall k\ge1, \quad |\lambda_k-(C_1+C_2k^2)|\le rk.
    \end{equation}
    We additionally assume that \begin{equation}
        \forall k\ge1, \quad \lambda_k> C_1\label{cond lambd2}.
    \end{equation}
    \item The first value $\lambda_0$ is isolated from the rest of the family, in the sense that it verifies \begin{equation}
        \lambda_0<C_1.\label{cond lambd3}
    \end{equation}
\end{itemize}
For $T>0$, we then wish to build a bi-orthogonal family $(q_j)_{j\ge 1}\subset L^2(0,T)$ verifying\begin{align}\label{cond bior}
    \forall j\ge 1, \forall k\ge 0, \quad \int_{0}^{T}q_j(t)e^{\lambda_kt}\,\dd t=\delta_{j,k}.
\end{align}
Note that all those functions must in particular be orthogonal to $t\mapsto e^{\lambda_0 t}$ in $L^2(0,T)$.\\
Under the spectral assumptions \eqref{cond lambd1}, \eqref{cond lambd2} and \eqref{cond lambd3}, we obtain the following result.

\begin{thm}\label{thm bior}
    For any $\kappa>1$, there exists $c>0$ and a family $(q_j)_{j\ge 0}\subset L^2\left(0,T\right)$ satisfying the constraint \eqref{cond bior} as well as the estimates:\begin{equation}\label{maj q}
        \forall j\ge 1, \quad \|q_j\|_{L^2\left(0,T\right)} \le c \left(\lambda_j-C_1\right)^{-1}\exp\left(-\frac{C_1T}{2} + \frac{3\pi^2\kappa}{2C_2T}\right).
    \end{equation}
\end{thm}

\begin{proof}
    We proceed as follows. We begin by absorbing the dissipation term $C_1$. Following the approach of Lissy in \cite{lissy2012link}, we look for such a family $(q_j)_{j\ge 0}$ of the form: \begin{equation}\label{bior q}
    \forall j\ge 1, \quad q_j(t)=\begin{cases}0, & t\in\left[0,\frac T2\right],\\
    e^{-C_1t}\exp\left(-\frac{3T}{4} (\lambda_j-C_1)\right) p_j\left(t-\frac{3T}{4}\right),& t\in\left[\frac T2,T\right],\end{cases}
\end{equation}
where the new family $(p_j)_{j\ge 0}\in L^2\left(-\frac T4,\frac T4\right)$ verifies, for $j\ge1, k\ge 0,$\begin{equation}\label{bior pi}
    \int_{-\frac T4}^{\frac T4} p_j(t)e^{(\lambda_k-C_1)t}\,\dd t=\delta_{j,k}.
\end{equation}
This condition is equivalent to the constraint \eqref{cond bior} on the functions $q_j$.\\
Define the family $(\mu_k)_{k\ge -n}$ of distinct real numbers :\begin{equation}
    \forall k\ge 0, \quad \mu_k:=\lambda_k-C_1.
\end{equation}
The assumptions \eqref{cond lambd1}, \eqref{cond lambd2} and \eqref{cond lambd3} imply:\begin{itemize}
    \item The $(\mu_k)_{k\ge1}$ form a family of distinct positive numbers such that:\begin{equation}\label{cond mu1}
        \forall k\ge 1, |\mu_k-C_2k^2|\le rk.
    \end{equation}
    \item The first value $\mu_0$ is isolated from the rest of the spectrum, as:\begin{equation}\label{cond mu2}
        \mu_0<0.
    \end{equation}
\end{itemize}
To build the family $(p_j)_{j\ge 0}$, we follow the same construction as Tenenbaum and Tucsnak in \cite{tenenbaum2007new}, with an adaptation to not re-excite the first eigenmode. 
\begin{Lemma}\label{lem tt}
    Let $(\mu_k)_{k\ge 0}$ verifying the assumptions above. For any $\kappa>1$, there exists $c>0$ and a family $(p_j)_{j\ge 1}$ satisfying the constraints \eqref{bior pi} as well as the estimates:\begin{equation}\label{maj p}
        \forall j\ge 1, \quad \|p_j\|_{L^2\left(-\frac T4, \frac T4\right)}\le c \mu_j^{-1} \exp\left(\frac{3\pi^2\kappa}{2C_2T}\right).
    \end{equation}
\end{Lemma}
We emphasize that, despite having the extra condition that the functions $p_j$ must be orthogonal to $t\mapsto e^{\mu_0 t}$ in $L^2\left(-\frac T4,\frac T4\right),$ we retrieve the exact same exponential term as \cite{tenenbaum2007new} in the upper bound of our bi-orthogonal family.
\begin{proof}[Proof of Lemma \ref{lem tt}]
    Without loss of generality, we assume that $C_2=1$. We may always go back to this setting by performing a linear change of variable.\newline\newline
    The standard method for the construction of such families, as done in \cites{glass2010complex, tenenbaum2007new,fattorini1974uniform}, is to build the bi-orthogonal functions as inverse Fourier transforms of entire functions that vanish at all the $(i\mu_k)_{k\ge 1}$ but one. To apply the Paley-Wiener theorem and recover suitable estimates on the functions, we need to construct such entire functions that are of exponential type, and their restriction to the real axis lies in $L^2$.\newline\newline
    Let $\Phi$ the entire function defined by
    \begin{equation}
        \Phi(z):=\prod_{k\ge 1}\left(1+i\frac z{\mu_k}\right).
    \end{equation}
    From the assumption \eqref{cond mu1}, this function is well defined and of exponential type. We may then define a family of entire functions $(\Phi_k)_{k\ge 1}$ as follows:\begin{equation}
        \Phi_k(z):=\frac{\Phi(z)}{\Phi'(i\mu_k) (z-i\mu_k)}.
    \end{equation}
    Using the repartition of the $(\mu_k)_{k\ge1}$ given by \eqref{cond mu1}, we obtain the following estimates on $\Phi_k$. There $B\ge 0, c>0$ two constants independent of $k$ such that:\begin{align}\label{phi C}
        \forall k\ge 1,\forall z\in\C,\quad |\Phi_k(z)|&\le c (1+|z|)^B \exp\left(\pi\sqrt{|z|}\right),\\\label{phi R}
        \forall k\ge 1,\forall x\in\R, \quad |\Phi_k(x)|&\le c (\mu_k+|x|)^B \exp\left(\pi\sqrt{|x|/2}\right).
    \end{align}
    Moreover, by construction, the entire functions $\Phi_k$ verify
    \begin{equation}
        \forall j\ge 1,k\ge1, \quad \Phi_j(i\mu_k)=\delta_{j,k}.
    \end{equation}
    We also want this constraint to hold for $k=0$, and therefore correct the functions $\Phi_j$ as follows. For $j\ge 1$, soit 
    \begin{equation}
        \wt{\Phi_j}(z):=\Phi_j(z)\cdot \frac{-iz-\mu_0}{\mu_j-\mu_0}.
    \end{equation}
    The entire functions $\wt{\Phi_j}$ then verify, as desired: \begin{equation}\label{crit phii}
        \forall j\ge 1,\forall k\ge 0, \quad \wt{\Phi_j}(i\mu_k)=\delta_{j,k}.
    \end{equation}
    Moreover, we corrected the entire functions by a polynomial factor, which does not affect their exponential type, and the term $(\mu_j-\mu_0)^{-1}$ is bounded uniformly in $j$ through the assumption \eqref{cond mu2}. The new entire functions $\wt{\Phi_j}$ therefore verify the same growth estimates \eqref{phi C} and \eqref{phi R}, up to replacing the constant $B$ by $B+1$.\\
    
    We now define the same multiplier function as in \cite{tenenbaum2007new}, which aims at compensating the growth of the $\wt{\Phi_j}$ on the real axis.\\
    For $\beta>0, \delta>0$, let \[\nu=\frac{(\pi+\delta)^2}{\beta},\] and define the entire function $H_\beta$:\begin{equation}
    H_\beta(z):=C_\nu \int_{-1}^1 \exp\left(-\frac{\nu}{1-t^2}-i\beta tz\right)\,\dd t,
\end{equation}
where $C_\nu$ is a normalization constant so that $H_\beta(0)=1$.
The entire function $H_\beta$ satisfies: \begin{align}
    |H_\beta(iy)|&\ge\frac1{11\sqrt{\nu+1}}\exp\left(\frac{\beta|y|}{2\sqrt{\nu+1}}\right),\quad &\forall y\in\R,\label{hbet imi}\\
    |H_\beta(z)|&\le e^{\beta|z|},\quad &\forall z\in\C,\label{croiss hbeti}\\
    |H_\beta(x)|&\le \sqrt{\nu+1}\exp\left(\frac{3\nu}4-(\pi+\frac\delta2)\sqrt{|x|}\right),\quad &\forall x\in\R.\label{hbet rei}
\end{align}
We choose the multiplier function \begin{equation*}
    f(z):=H_\beta\left(\frac z2\right).
\end{equation*}
We may then obtain the desired entire functions $J_j$:
\begin{equation}
    \forall j\ge 0, z\in\C, \quad J_j(z):= \wt{\Phi_j}(z) \cdot\frac{f(z)}{f(i\mu_j)}.
\end{equation}
By combining the upper bounds \eqref{phi R}, \eqref{hbet imi} and \eqref{hbet rei}, it follows: \begin{equation}
    |J_j(x)|\le c \exp\left(-\frac{\beta\mu_j}{4\sqrt{\nu+1}} + \frac{3\nu}{4} - \frac\delta{2\sqrt2}\sqrt{|x|}\right).
\end{equation}
We then choose $\beta\in\left(\frac {T}{2\kappa},\frac T2\right)$ and $\delta>0$ such that \[\nu<\frac{(4-\delta)\pi^2\kappa}{2T}.\] 
It immediately yields, for $j\ge 1$: \begin{equation}
    \|J_j\|_{L^2(\R)}\le c(\kappa) (\mu_j)^{-1} \exp\left(\frac{3\pi^2\kappa}{2T}\right).
\end{equation}
From the growth estimates \eqref{phi C} and \eqref{croiss hbeti}, we are in a suitable context to apply the Paley-Wiener theorem, and there exists a family $(p_j)_{j\ge0}$ in $L^2(-\frac T4,\frac T4)$, such that \begin{equation}
    J_j(z)=\int_{-\frac T4}^{\frac T4} p_j(t)e^{-itz}\,\dd t,
\end{equation}
as well as \[\|p_j\|_{L^2\left(-\frac T4,\frac T4\right)}\le c (\mu_j)^{-1}\exp\left(\frac{3\pi^2\kappa}{2T}\right).\]\\
Moreover, the relation \eqref{crit phii} implies that $J_j(i\mu_k)=\delta_{j,k}$ pour $j\ge1,k\ge0$, and the family $(p_j)_{j\ge_0}$ thus verifies the bi-orthogonality condition \eqref{bior pi}.\\
This concludes the proof \ref{lem tt}.
\end{proof}
Equipped with this result, we resume the proof of Theorem \ref{thm bior}. Let $\kappa>1$ and $(p_j)_{j\ge 0}\subset L^2\left(-\frac T4,\frac T4\right)$ a family given by \ref{lem tt}. We recall the definition \eqref{bior q} of the bi-orthogonal family $(q_j)_{j\ge 1}\subset L^2(0,T)$. Since $C_1$ is positive, we directly obtain, for all $j\ge 1$, the upper bound:\begin{equation}
    \|q_j\|_{L^2(0,T)} \le \exp\left(-C_1\frac T2\right) \|p_j\|_{L^2\left(-\frac T4,\frac T4\right)}.
\end{equation}
Finally we use the estimate \eqref{maj p} and, recalling that $\mu_k=\lambda_k-C_1$, it directly leads to the  desired upper bound \eqref{maj q}, which concludes the proof of Theorem \ref{thm bior}.
\end{proof}

    \bibliography{Ref}
    \bibliographystyle{alpha}
    
\end{document}